\documentclass[11pt, reqno]{amsart}
\usepackage{mathrsfs}
\usepackage[active]{srcltx}
\usepackage{mathrsfs}
\usepackage{longtable}

\usepackage[T1]{fontenc}
\usepackage[latin9]{inputenc}
\usepackage[a4paper]{geometry}
\usepackage{amsthm}
\usepackage{amssymb}
\usepackage{mathtools}
\usepackage{mathrsfs}
\usepackage{xcolor}
\usepackage{amsmath,stmaryrd,graphicx}
\usepackage{hyperref}
\usepackage{enumerate}
\usepackage{titlesec}
\titlelabel{\thetitle.\quad}

\usepackage{graphicx}%
\usepackage{color}
\usepackage{cite}
\usepackage{hyperref}
\hypersetup{
	colorlinks = true,%
	citecolor = blue,%[rgb]{0.65,0.0,0.0},
	filecolor=red,%
	linkcolor = [rgb]{0.65,0.0,0.0},%
	anchorcolor = red,
	pagecolor = red,
	urlcolor= [rgb]{0.65,0.0,0.0}
	linktocpage=true,
	pdfpagelabels=true,
	bookmarksnumbered=true,
}

\numberwithin{equation}{section}

\newtheorem{theorem}{Theorem}[section]
\newtheorem{lemma}{Lemma}[section]

\theoremstyle{definition}
\newtheorem{remark}{Remark}[section]
\newtheorem{proposition}{Proposition}[section]
\begin{document}

\title[Sobolev  inequalities on Riemann-Finsler manifolds]
{Sharp Morrey-Sobolev inequalities and eigenvalue problems on Riemannian-Finsler manifolds with nonnegative Ricci curvature}
%{Anisotropic symmetrization and Sobolev inequalities on Finsler manifolds with nonnegative Ricci curvature
%}
%{Morrey-Sobolev and Hardy inequalities on Finsler manifolds with nonnegative Ricci curvature}

\date{}

\author{Alexandru Krist\'aly \and \'Agnes Mester \and Ildik\'o I. Mezei}

\thanks{A. Krist\'aly  is supported by the UEFISCDI/CNCS grant PN-III-P4-ID-PCE2020-1001 (Romania) and Excellence Scholarship Program  \'OE-KP-2-2022 (Hungary).  
	I.I. Mezei is supported by the  UEFISCDI/CNCS grant PN-III-P4-ID-PCE2020-1001 (Romania). \'A. Mester is supported by  the National Research, Development and Innovation Fund of Hungary, financed under the K$\_$18 funding scheme, Project No.  127926 (Hungary).
}

\address{A. Krist\'aly: Institute of Applied Mathematics, \'Obuda
	University, B\'ecsi \'ut 96/B, 1034
	Budapest, Hungary \& Department of Economics, Babe\c s-Bolyai University, Str. Teodor Mihali 58-60, 400591 Cluj-Napoca,
Romania}

\email{kristaly.alexandru@nik.uni-obuda.hu; alexandru.kristaly@ubbcluj.ro}

\address{\'A. Mester: 
Department of Mathematics and Computer Sciences, Babe\c s-Bolyai University, 
Str. Mihail Kog\u alniceanu 1, 
400084 Cluj-Napoca, Romania 
\& 
Institute of Applied Mathematics, \'Obuda	University, B\'ecsi \'ut 96/B, 1034 Budapest, Hungary
}
\email{agnes.mester@ubbcluj.ro}

\address{I. Mezei: Department of Mathematics and Computer Sciences, Babe\c s-Bolyai University, Str. Mihail Kog\u alniceanu 1, 400084 Cluj-Napoca,
Romania}
\email{ildiko.mezei@ubbcluj.ro}

\begin{abstract} 
	\noindent 	Combining the sharp isoperimetric inequality established by Z. Balogh and A. Krist\'aly [\textit{Math.\ Ann.}, in press,  doi.org/10.1007/s00208-022-02380-1] with an anisotropic symmetrization argument, we establish sharp Morrey-Sobolev  inequalities on $n$-dimensional Finsler manifolds having nonnegative $n$-Ricci curvature.\
	 A byproduct of this method is  a	 Hardy-Sobolev-type inequality in the same geometric setting. 
	As applications, by using variational arguments, we guarantee the existence/multiplicity of solutions for certain eigenvalue problems and elliptic PDEs involving the Finsler-Laplace operator. Our results are also new in the Riemannian setting. 
\end{abstract}

	%\vspace*{0.3cm}
	
	
	
%	\noindent 
%	{\normalfont\textbf{2020 Mathematics Subject Classification:} }

\maketitle

%\tableofcontents
	
\vspace{-0.8cm}
\section{Introduction}

Most of elliptic PDEs are studied over Sobolev spaces which are usually embedded into certain Lebesgue spaces; this fact is  described quantitatively by Sobolev  inequalities. Within this theory, a prominent class of Sobolev inequalities is provided by those defined on curved structures. Motivated mainly by the Yamabe problem, Aubin \cite{Aubin} initiated in the early seventies the so-called \textit{AB-program}, i.e., to determine the  best constants within such Sobolev inequalities  on Riemannian manifolds. It turned out that this study deeply depends on the curvature of the ambient space, and it is still a very active area of geometric analysis. A comprehensive work in this topic is provided by Hebey \cite{Hebey} and subsequent references.

Roughly speaking, two main classes of Sobolev inequalities can be distinguished, depending on the curvature restriction of  (noncompact) Riemannian manifolds, having: (a) nonpositive sectional curvature, or (b) Ricci curvature bounded from below. 

In case (a), Sobolev inequalities similar to Euclidean ones hold on Cartan-Hadamard manifolds\footnote{Complete, simply connected Riemannian manifolds with nonpositive sectional curvature.}, having the same sharp constants   as in their Euclidean counterparts, see e.g.\ Druet,  Hebey and  Vaugon \cite{Druetetal}, Hebey \cite{Hebey} and Muratori and  Soave \cite{MS}. One of the main tools to prove such Sobolev inequalities is a Schwarz-type symmetrization 'from manifolds to Euclidean spaces', combined with sharp isoperimetric inequalities, known as the Cartan-Hadamard conjecture, which is valid in low-dimensions; see e.g.  
 Ghomi and Spruck \cite{GS}, and Kloeckner and   Kuperberg \cite{KK}. 

In case (b), the existence of a lower bound for the volume of small balls which is uniform w.r.t.\ their center  characterizes   the validity of Sobolev inequalities, see Hebey \cite[Chapter 3]{Hebey}. In the particular case when the Ricci curvature is nonnegative,  rigidity phenomena occur, i.e., a Sobolev inequality with the same Sobolev constant as in its Euclidean counterpart is supported on such  a manifold if and only if the manifold is isometric to the Euclidean space, see Ledoux \cite{Ledoux-CAG}. Quantitatively speaking, a close constant in a Sobolev inequality to its optimal Euclidean counterpart implies 'topologically closer' manifold to the Euclidean space, described by the trivialization of homotopy groups, see  do Carmo and Xia \cite{doCarmo-Xia} (and its nonsmooth version for $CD(0,N)$ spaces in Krist\'aly \cite{Kristaly-Calculus}). 

Very recently, Balogh and Krist\'aly \cite{BaloghKristaly} proved sharp $L^p$-Sobolev inequalities on $n$-dimensional Riemannian manifolds having nonnegative Ricci curvature 
%(for short, ${\sf Ric}\geq 0$) 
and Euclidean volume growth, whenever $1<p<n$. They used symmetrization arguments and a sharp isoperimetric inequality  recently proved by Brendle \cite{Brendle}, and alternatively, by themselves \cite{BaloghKristaly}. We notice that the sharp isoperimetric inequality in \cite{BaloghKristaly} is valid even for generic $CD(0,N)$ spaces, thus in particular, for reversible Finsler manifolds with nonnegative $n$-Ricci curvature (for short, ${\sf Ric}_n\geq 0$, see \S \ref{section-2}).  

 Based on the sharp isoperimetric inequality in \cite{BaloghKristaly},  the main purpose of the present paper is to adapt a suitable Schwarz-type symmetrization argument to Finsler manifolds with ${\sf Ric}_n\geq 0$ and to establish (possibly sharp) functional inequalities. In particular,  we aim to handle 
the complementary case $p>n$ w.r.t. \cite{BaloghKristaly} by establishing sharp Morrey-Sobolev inequalities on such Finsler structures. We remark that our results are also new in the Riemannian setting.

 In order to give a flavor of our results, let $(M, F)$ be a noncompact, complete $n$-dimensional reversible Finsler manifold with ${\sf Ric}_n\geq 0$, endowed with the canonical volume form $\mathrm{d}v_F$ and the induced  Finsler metric $d_F: M \times M \to \mathbb{R}$, and let $F^*$ be the polar transform of $F$; for these notions, see \S \ref{section-2}. Let $B_x(r) = \{ z \in M: d_F(x,z)<r \}$ be the geodesic ball with center $x\in M$ and radius $r>0$. The \textit{asymptotic volume ratio of} $(M, F)$ is defined as
\begin{equation*}
\mathsf{AVR}_F = \lim_{r \to \infty} \frac{\mathrm{Vol}_F(B_x(r))}{\omega_n r^n},
\end{equation*}
where $\mathrm{Vol}_F(S) = \displaystyle\int_{S} \mathrm{d}v_F$ for any measurable set $S\subset M$, while $\omega_n = \pi^\frac{n}{2}/ \Gamma(1+\frac{n}{2})$ denotes the volume of the Euclidean open unit ball in $\mathbb{R}^n$. 
Note that $\mathsf{AVR}_F\in [0,1]$ is well defined, i.e., it is independent of the choice of $x \in M$. Also, by the generalized Bishop-Gromov volume growth inequality, see Shen \cite{Shen-volume}, we have that the mapping $r \mapsto \frac{\mathrm{Vol}_F(B_x(r))}{r^n}$ is nonincreasing on $(0,\infty)$ for every $x \in M$. We say that  $(M,F)$ has \textit{Euclidean volume growth} whenever $\mathsf{AVR}_F > 0$. 

The key result is a \textit{P\'olya-Szeg\H o inequality}  on Finsler manifolds with ${\sf Ric}_n\geq 0$, involving the asymptotic volume ratio $\mathsf{AVR}_F$ (see Theorem \ref{th:Polya-Szego}), whose  proof is based on the sharp isoperimetric inequality from \cite{BaloghKristaly} and a  symmetrization argument in the spirit of Aubin \cite{Aubin}. Although the latter symmetrization is well known,  a careful adaptation is needed 'from manifolds to normed spaces'.  
The first main consequence is  the following sharp Morrey-Sobolev inequality with \textit{support-bound}: 

\begin{theorem}\label{theorem-1-intro}
	Let $(M, F)$ be a noncompact, complete $n$-dimensional reversible Finsler manifold with ${\sf Ric}_n\geq 0$ and Euclidean volume growth $0 < \mathsf{AVR}_F \leq 1$. If $p > n \geq 2$, then  one has 
	\begin{equation}\label{Morrey-Sobolev-Finsler-intro}
		\| u\|_{L^\infty(M)} \leq{\sf T}_{F}^{\sf MS} \, {\rm Vol}_{F}({\rm supp}\, u)^{\frac{1}{n}-\frac{1}{p}}\left(\int_{M} F^*(x, D u(x))^p \mathrm{d}v_F \right)^{\frac{1}{p}}  ,\ \ \ \forall u \in C_0^\infty(M),
	\end{equation}
	where the constant 
	$$ {\sf T}_{F}^{\sf MS}=n^{-\frac{1}{p}}\omega_n^{-\frac{1}{n}}\left(\frac{p-1}{p-n}\right)^\frac{p-1}{p} {\sf AVR}_F^{-\frac{1}{n}}$$ is  sharp.
\end{theorem}

 A counterpart of Theorem \ref{theorem-1-intro} is the following sharp Morrey-Sobolev inequality with \textit{$L^1$-bound}: 

\begin{theorem} \label{Th_Morrey1-0}
	Let $(M, F)$ be a noncompact, complete $n$-dimensional reversible Finsler manifold with  ${\sf Ric}_n\geq 0$ and $0 < \mathsf{AVR}_F \leq 1$, and for any $p>n\geq 2$, consider the constant $\eta = \frac{np}{np+p-n}$.  Then, for every $u \in C_0^\infty(M)$ one has 
	\begin{equation}\label{Morrey_Sobolev_Finsler-0}
		\|u\|_{L^\infty(M)} \leq
		\mathsf{C}_{F}^{\sf MS} \left(\int_M |u(x)| \mathrm{d}v_F \right)^{1-\eta}
		\left(\int_M F^*(x, D u(x))^p \mathrm{d}v_F \right) ^{\frac{\eta}{p}}, 
	\end{equation}
	where the constant 
	$$
	\mathsf{C}_{F}^{\sf MS}  =  (n\omega_n^\frac{1}{n})^{-\frac{np'}{n+p'}}\left(\frac{1}{n}+\frac{1}{p'}\right)
	\left(\frac{1}{n}-\frac{1}{p}\right)^\frac{(n-1)p'-n}{n+p'}
	\left({\sf B}\left(
	\frac{1-n}{n}p'+1,p'+1\right)\right)^\frac{n}{n+p'} {\sf AVR}_F^{-\frac{\eta}{n}}
	$$
	is sharp. Hereafter, $p' = \frac{p}{p-1}$, and $\sf B(\cdot, \cdot)$ denotes the Euler beta-function.
\end{theorem}

Theorems \ref{theorem-1-intro} and   \ref{Th_Morrey1-0} are also new in the Riemannian setting and can be viewed as  
new pieces within the aforementioned \textit{AB-program} of Aubin \cite{Aubin}. 
 Note that Krist\'aly \cite{Kristaly-Potential} proved that whenever \eqref{Morrey-Sobolev-Finsler-intro} and \eqref{Morrey_Sobolev_Finsler-0} hold with generic constants instead of	$\mathsf{C}_{F}^{\sf MS}$ and $ {\sf T}_{F}^{\sf MS}$, then there is a non-collapsing phenomenon of the metric  balls on the Riemannian manifold. In the Euclidean case (when the asymptotic volume ratio is 1), Theorems \ref{theorem-1-intro} and   \ref{Th_Morrey1-0} reduce to well known results of Talenti \cite{Talenti}. A natural question arises: are there nonzero extremal functions in \eqref{Morrey-Sobolev-Finsler-intro} and \eqref{Morrey_Sobolev_Finsler-0}? 
At this moment, we can provide such an answer within the class of Riemannian manifolds: equality occurs in \eqref{Morrey-Sobolev-Finsler-intro} and \eqref{Morrey_Sobolev_Finsler-0} for a nonzero, enough smooth function if and only if the Riemannian manifold is isometric to the Euclidean space $\mathbb R^n$, see Theorem \ref{Th_Morrey_Rigidity}.

%The proof of Theorem \ref{theorem-1-intro} is based on a particular symmetrization argument 'from manifolds to normed spaces' -- by proving a P\'olya-Szeg\H o-type inequality, -- combined with the aforementioned sharp isoperimetric inequality from \cite{BaloghKristaly}. 

A natural byproduct of our  arguments is the validity of Sobolev inequalities involving \textit{singular terms} within the same geometric setting as above; in fact, we can prove  the following Hardy-Sobolev-type inequality:

 \begin{theorem}\label{theorem-Hardy} Let $(M, F)$ be a noncompact, complete $n$-dimensional reversible Finsler manifold with ${\sf Ric}_n\geq 0$ and Euclidean volume growth $0 < \mathsf{AVR}_F \leq 1$. If $n>p>1$, 
 	 we have for every $x_0 \in M$ that
 	\begin{equation}\label{Hardy-intro}
 		\int_{M} F^*(x, D u(x))^p \mathrm{d}v_F \geq 
 		\mathsf{AVR}_F^\frac{p}{n} \left(\frac{n-p}{p}\right)^p \int_{M} \frac{|u(x)|^p}{d_F(x_0, x)^p} \mathrm{d}v_F, \ \ \ \forall u \in C_0^\infty(M). 
 	\end{equation}
\end{theorem}
\noindent As before, Theorem \ref{theorem-Hardy} is also new in the Riemannian framework; although expected, we do not know the sharpness of \eqref{Hardy-intro} unless we are in the Euclidean setting. Theorem \ref{theorem-Hardy} can be viewed as a counterpart of Hardy-Sobolev inequalities on Cartan-Hadamard manifolds, established e.g. by Berchio, Ganguly,  Grillo and Pinchover \cite{Berchio-etal}, D'Ambrosio and Dipierro \cite{olaszok},   Huang, Krist\'aly and Zhao \cite{HKZ}, Krist\'aly \cite{Kristaly-JMPA}, and Zhao \cite{Zhao}. An improved version of Theorem \ref{theorem-Hardy} for $p=2$ is stated in Theorem \ref{Th:BrezisPoincareVazquez}, which is a Brezis-Poincar\'e-V\'azquez inequality on Finsler manifolds.  

The second purpose of the paper -- which is not detailed here -- is to show the applicability of the aforementioned functional inequalities. First, by using Theorem \ref{theorem-1-intro} and a variational argument \`a la Ricceri \cite{Ricceri00},  in Theorem \ref{Appl:1} we provide a multiplicity result for an elliptic PDE involving the $p$-Finsler-Laplace operator. Second, by applying the Hardy-Sobolev inequality (Theorem \ref{theorem-Hardy}) and the  Brezis-Poincar\'e-V\'azquez inequality (Theorem \ref{Th:BrezisPoincareVazquez}), we provide a sufficient condition  to guarantee the existence of a positive solution for a sub-critical elliptic PDE  involving the 2-Finsler-Laplace operator, see Theorem \ref{Application2}. 
%Moreover, in the Riemannian setting,  the sufficient condition turns out to be also necessarily concerning the existence of such solutions, see Corollary \ref{corollary}.  

Finally, we emphasize the richness of those Finsler manifolds where our results can be applied. Beside Riemannian manifolds with nonnegative Ricci curvature and Euclidean volume growth (extensively studied in \cite{BaloghKristaly}),  we provide a whole class of non-Riemannian Finsler manifolds modeled over $\mathbb R^n$ with the required properties. More precisely,  we endow the space $\mathbb R^{n-1}$ ($n \geq 3$) with a complete Riemannian
metric $g$ having
nonnegative Ricci curvature and assume that the induced warped metric $\tilde g$ on $\mathbb R^{n-1}\times \mathbb R$, defined by $\tilde g_{(x,t)}(v,w)=\sqrt{g_x(v,v)+w^2},\ \  (x,t)\in \mathbb R^{n}, (v,w)\in T_{x}\mathbb R^{n-1}\times  T_t\mathbb R,$
has Euclidean volume growth $0 < \mathsf{AVR}_{\tilde g} \leq 1$.  Then the parameter-depending Finsler manifold $(\mathbb R^n,F_\varepsilon)$, with $\varepsilon>0$, defined by means of the Finsler metric 
$F_\varepsilon:T\mathbb R^{n}\rightarrow [0,\infty)$, 
\[ F_\varepsilon((x,t),(v,w))=\sqrt{g_x(v,v)+w^2 +
	\varepsilon \sqrt{g_x(v,v)^2+w^4}},\ \  (x,t)\in \mathbb R^{n}, (v,w)\in T_{x}\mathbb R^{n-1}\times  T_t\mathbb R, \]
has the properties that  ${\sf Ric}_n\geq 0$ and  $0 < \mathsf{AVR}_{F_\varepsilon} \leq 1$; for details, see \S \ref{subsection-example}.

The paper is organized as follows. In Section \ref{section-2} we recall those auxiliary notions that are used throughout the whole paper. Section \ref{section-symmetrization} is devoted to the anisotropic symmetrization, by proving among others, a P\'olya-Szeg\H o inequality involving the asymptotic volume ratio $\mathsf{AVR}_F$ together with a Hardy-Littlewood-P\'olya inequality. In Section \ref{sec:Sobolev_inequalities} we prove Morrey-Sobolev and Hardy-Sobolev inequalities on Finsler manifolds with ${\sf Ric}_n\geq 0$ and $\mathsf{AVR}_F\in (0,1]$,  and discuss the geometric properties of the non-Riemannian Finsler manifolds $(\mathbb R^n,F_\varepsilon)$ introduced in the previous paragraph. Applications of Sobolev inequalities  are presented in Section \ref{section-PDE} for two elliptic PDEs, both involving the Finsler-Laplace operator.

\section{Preliminaries on Finsler geometry} \label{section-2}

Let $M$ be a connected $n$-dimensional smooth manifold and $TM=\bigcup_{x \in M}T_{x} M $ its tangent bundle. 
The pair $(M,F)$ is called a Finsler manifold if
$F: TM \to [0,\infty)$ is a continuous function such that
\begin{enumerate}[(i)]
	\item $F$ is $C^{\infty}$ on $TM\setminus\{ 0 \}$; 
	\item $F(x,\lambda y) = \lambda F(x,y)$ for all $\lambda \geq 0$ and all $(x,y)\in TM$;
	\item the $n \times n$ Hessian matrix $\Big( g_{ij}(x,y) \Big) := \left( \left[ \frac{1}{2}F^{2}(x,y)\right] _{y^{i}y^{j}} \right)$
	is positive definite for all $(x,y)\in TM\setminus \{0\}.$
\end{enumerate}
If, in addition, $F(x,\lambda y) = |\lambda| F(x,y)$ holds for every $\lambda \in \mathbb{R}$ and  $(x,y) \in TM$, then the Finsler manifold is called reversible.

A curve $\gamma: [0,r] \to M$ is called a geodesic if its velocity field $\dot \gamma$ is parallel along the curve, i.e., $D_{\dot \gamma} \dot \gamma = 0$, where $D$ denotes the covariant derivative induced by the Chern connection, see Bao, Chern and Shen \cite[Chapter~2]{BCS}.   
$(M,F)$ is said to be complete if every geodesic $\gamma: [0,r] \to M$ can be extended to a geodesic defined on $\mathbb R$.

The distance function $d_F: M \times M \to [0, \infty)$ is defined by
$$ d_F(x_1, x_2) = \inf_{\gamma \in \Lambda(x_1,x_2)}  \int_0^r F( \gamma(t), \dot \gamma(t))  {\mathrm d}t,$$
where $\Lambda(x_1,x_2)$ denotes the set
of all piecewise $C^{\infty}$ curves $\gamma:[0,r] \to M$ such that
$\gamma(0)=x_1$ and $\gamma(r)=x_2$. 
Clearly, $d_{F}(x_1,x_2) =0$ if and only if $x_1=x_2$, and $d_F$ verifies the triangle inequality, as well.
However, $d_F$ is symmetric if and only if $(M,F)$ is a reversible Finsler manifold.

Let
$ B_x(1) = \Big\{ (y^i) \in \mathbb{R}^n :~ F \Big(x, \sum_{i=1}^n y^i \frac{\partial}{\partial x^i} \Big) < 1 \Big\} \subset \mathbb{R}^n ,$
and define the ratio
$$ \sigma_F(x) = \frac{ \omega_n}{\mathrm{Vol}(B_x(1))} ,$$
where $\mathrm{Vol}$ denotes in the sequel the canonical Euclidean volume, and $\omega_n = \pi^\frac{n}{2}/ \Gamma(1+\frac{n}{2})$ is the volume of the  $n$-dimensional Euclidean open unit ball.
The Busemann-Hausdorff volume form is defined as  
\begin{equation} \label{Hausdorff_measure}
	{\mathrm d}v_F(x) = \sigma_F(x) d x^1 \land \dots \land d x^n ,
\end{equation}
see Shen \cite[Section~2.2]{Shen01}. 
%Note that in the following we may omit the parameter $x$ for the sake of brevity. 

For a fixed point $x \in M$ let $y,v \in T_x M$ be two linearly independent tangent vectors.
Then the flag curvature is defined as
$$ \mathrm{K}^y(y, v) = \frac{g_y(\mathrm{R}(y,v)v,y)}{ g_y(y,y) g_y(v,v) - g_y(y,v)^2} ,$$ 
where $g$ is the fundamental tensor induced by the Hessian matrices $(g_{ij})$, and $\mathrm{R}$ is the Chern curvature tensor, see Bao, Chern and Shen \cite[Chapter~3]{BCS}. 
The Ricci curvature is defined as
$$ \mathsf{Ric}(y) = F^2(x,y) \sum_{i=1}^{n-1} \mathrm{K}^y(y, e_i),$$
where $\{ e_1, ... ,e_{n-1}, \frac{y}{F(x,y)} \}$ is an orthonormal basis of $T_xM$ with respect to $g_y$.

Let
$\{e_i\}_{i=1,...,n}$ be a basis for $T_xM$. The  mean distortion $\mu:TM\setminus
\{0\}\to (0,\infty)$ is defined by $\mu(x,y)=\frac{{\rm
Vol}(B_x(1))}{\omega_n}\sqrt{{\rm det}(g_{ij}(x,y))}$. The mean
covariation $\mathcal H:TM\setminus\{0\}\to \mathbb R$ is defined by
$\mathcal H(x,y)=\frac{d}{dt}(\ln \mu(\gamma_{(x,y)}(t),\dot\gamma_{(x,y)}(t)))|_{t=0},$ where
$\gamma_{(x,y)}$ is the geodesic such that $\gamma_{(x,y)}(0)=x$ and $\dot
\gamma_{(x,y)}(0)=y.$ We say that $(M,F)$ has nonnegative $n$-Ricci curvature, denoted by ${\sf Ric}_n\geq 0$, if $\mathsf{Ric}\geq 0$ and  the mean
covariation  $\mathcal H$ is identically zero. Note that  Finsler manifolds of Berwald type (i.e., the coefficients of the Chern connection $\Gamma_{ij}^k(x,y)$ do not depend on $y\in T_xM$) endowed with the Busemann-Hausdorff measure have vanishing mean covariation, see Shen \cite{Shen-volume}; this class contains both Riemannian manifolds and Minkowski spaces.
	
%	We say that $(M,F)$ has nonnegative Ricci curvature and denote it by $\mathsf{Ric} \geq 0$ if $\mathsf{Ric}(y) \geq 0$ for every $(x,y) \in T M$.

Endowed with the canonical volume form $\mathrm{d}v_F$, the $n$-di\-men\-sional Finsler manifold $(M,F)$ verifies  for every $x\in M$ that
\begin{equation}\label{volume-comp-nullaban}
	\lim_{r\to 0^+}\frac{{\rm Vol}_F(B_x(r))}{\omega_n r^n}=1.
\end{equation}
In addition, if $(M,F)$ is a complete  Finsler manifold with ${\sf Ric}_n\geq 0$, the Bishop-Gromov  volume comparison principle states that 
$r\mapsto \frac{{\rm Vol}_F(B_x(r))}{r^n}$ is a nonincreasing function on $(0,\infty)$, see Shen \cite[Theorem 1.1]{Shen-volume}. In particular, from the latter property and 
\eqref{volume-comp-nullaban}, it turns out that $\mathsf{AVR}_F\in [0,1]$.

The polar transform $F^*: T^*M \to [0, \infty)$ is defined as the dual metric of $F$, i.e.,
\begin{equation*}  
F^*(x,\alpha) = \sup_{y \in T_xM \setminus \{0\}} ~ \frac{\alpha(y)}{F(x,y)}.
\end{equation*}

Let $u: M \to \mathbb{R}$ be a differentiable function in the distributional sense. The gradient of $u$ is defined as
$
{\nabla}_F u(x) = J^*(x, Du(x)),
$
where $Du(x) \in T_x^*M$ denotes the (distributional) derivative of $u$ at $x\in M$ and $J^*$ is the Legendre transform given by 
$$J^*(x,\alpha) = \sum_{i=1}^n \frac{\partial}{\partial \alpha_i} \left(\frac{1}{2} F^{*2}(x,\alpha)\right) \frac{\partial}{\partial x^i}.$$  
In local coordinates, we have
\begin{equation*}  \label{derivalt-local}
Du(x)=\sum_{i=1}^n \frac{\partial u}{\partial x^i}(x)\mathrm{d}x^i \quad \text{and} \quad
{\nabla}_F u(x)=\sum_{i,j=1}^n g_{ij}^*(x,Du(x))\frac{\partial u}{\partial x^i}(x)\frac{\partial}{\partial x^j},
\end{equation*}
where $(g^*_{ij})$ is the Hessian matrix 
$\Big( g^*_{ij}(x,\alpha) \Big) = \left( \left[ \frac{1}{2}F^{*2}(x, \alpha)\right] _{\alpha^{i}\alpha^{j}} \right)$, see  Ohta and Sturm \cite[Lemma 1.1]{Ohta-Sturm}. 
Therefore, the gradient operator ${\nabla}_F$ is usually nonlinear. 

We recall the eikonal equation, i.e., if $x_0\in M$ is fixed, then by Ohta and Sturm \cite{Ohta-Sturm}, one has 
\begin{equation}  \label{tavolsag-derivalt}
F^*(x,D d_F(x_0,x))=F(x,{\nabla}_F d_F(x_0,x))=D d_F(x_0,x)(%
{\nabla}_F d_F(x_0,x))=1\ \mathrm{for\ a.e.}\ x\in M.
\end{equation}

If $X$ is a vector field on $M$, then, in local coordinates, the divergence of $X$ is defined as div$(X)=\frac{1}{\sigma_F}\frac{\partial}{\partial x^i}(\sigma_F X^i).$ 
The  $p$-Finsler-Laplace operator is given by $${\Delta}_{F,p}u(x) = \mathrm{div}\big({F^*}(x, Du(x))^{p-2} \cdot {\nabla}_F u(x)\big),$$
while the divergence theorem reads as
\begin{equation}  \label{Green}
\int_M v(x) {\Delta}_{F,p} u(x) \,{\mathrm d}v_F = - \int_M
{F^*}(x,Du(x))^{p-2} \cdot Dv(x)\big({\nabla}_F u(x)\big)\,{\text d}v_F, \quad \forall v \in C_0^\infty(M),
\end{equation}
see Ohta and Sturm \cite{Ohta-Sturm}. Note that in general, ${\Delta}_{F,p}$ is nonlinear. When $p=2$, the $2$-Finsler-Laplace operator is simply denoted by ${\Delta}_{F} \coloneqq {\Delta}_{F,2}$.   

Finally, let $\Omega$ be an open subset of $M$. 
The Sobolev space  on $\Omega$ associated with the Finsler structure $F$ is defined by  
$$W^{1,p}_F(\Omega)=\left\{u\in W^{1,p}_\mathrm{loc}(\Omega): \int_\Omega {F^*}(x,Du(x))^p ~\mathrm{d}v_F < +\infty \right\},$$
while $W_{0,F}^{1,p}(\Omega)$ is the closure of $C_{0}^{\infty
}(\Omega)$ with respect to the norm
\begin{equation*}
\|u\|_{W^{1,p}_F(\Omega)} = \left(\int_\Omega |u(x)|^p\,\mathrm{d}v_F + \int_\Omega {F^*}(x,Du(x))^p\,\mathrm{d}v_F \right)^{\frac{1}{p}}.
\end{equation*}
%see Ohta and Sturm \cite{Ohta-Sturm}.

\section{Anisotropic symmetrization on Finsler manifolds with $\mathsf{Ric}_n \geq 0$}\label{section-symmetrization}

In what follows, let $(M, F)$ be a noncompact, complete $n$-dimensional reversible Finsler manifold with ${\sf Ric}_n\geq 0$, endowed with the canonical volume form $\mathrm{d}v_F$ and the induced  Finsler metric $d_F: M \times M \to \mathbb{R}$. In particular, $(M,d_F,{\rm d}v_F)$ is a metric measure space satisfying the $CD(0,n)$ condition, see Ohta \cite{Ohta}. Consequently, by Balogh and Kristály \cite[Theorem 1.1]{BaloghKristaly}, for every bounded open set $\Omega \subset M$ with smooth boundary, we have the  sharp isoperimetric inequality
\begin{equation}\label{isoperimetric}
	\mathcal{P}_F(\partial \Omega) \geq n \omega_n^{\frac{1}{n}} \mathsf{AVR}_F^{\frac{1}{n}}\, \text{Vol}_F(\Omega)^{\frac{n-1}{n}}, 
\end{equation}
where $\mathcal{P}_F(\partial \Omega)$ stands for the anisotropic perimeter of $\partial \Omega$, defined as $\mathcal{P}_F(\partial \Omega) = \displaystyle\int_{\partial \Omega} \mathrm{d} \sigma_F$, where $\mathrm{d} \sigma_F$ stands for  the $(n-1)$-dimensional Lebesgue measure induced by $\mathrm{d}v_F$.

Beside the Riemannian setting, see Brendle \cite{Brendle} and Balogh and Krist{\'a}ly \cite{BaloghKristaly}, one can cha\-ra\-cterize the equality in \eqref{isoperimetric} in the case of the  simplest non-Riemannian Finsler structures, namely on Minkowski spaces. To be precise, let $(\mathbb{R}^n, H)$ be a Finsler manifold endowed with the Lebesgue measure $\mathrm{d}v_H$, such that $H: \mathbb{R}^n \to [0, \infty)$ is a smooth, absolutely homogeneous norm. A Wulff-shape associated to the norm $H$ is the set
\begin{equation} \label{Wulff-shape}
	W_H(R) \coloneqq \{ x \in \mathbb{R}^n : H(x) < R \},
\end{equation} 
for any number $R > 0$. In the sequel, we assume, without loss of generality, that the set $W_H(1)$
has measure $\mathrm{Vol}(W_H(1)) = \omega_n$; in this case, we say that $H:\mathbb R^n\to [0, \infty)$ is a normalized Minkowski norm. 
Accordingly, it turns out that $\sigma_H = 1$, i.e., $\mathrm{Vol}_H(W_H(1)) = \mathrm{Vol}(W_H(1)) = \omega_n$, which yields both ${\mathrm d}v_H(x)={\mathrm d}x$ in \eqref{Hausdorff_measure} and  $\mathsf{AVR}_H = 1$. 
Then, due to  Cabr\'e, Ros-Oton, and Serra \cite[Theorem 1.2]{CabreRosOton},  we have 
\begin{equation}\label{isoperimetric_equality}
	\mathcal{P}_H(\partial \Omega) = n \omega_n^{\frac{1}{n}}\, \text{Vol}_H(\Omega)^{\frac{n-1}{n}}, 
\end{equation}
if and only if $\Omega$ has a Wulff-shape, i.e., $\Omega = W_H(R)$ for some $R>0$ (up to translations).

Let $\mathscr{C}(M)$ be the space of continuous functions $u: M \to [0,\infty)$ with compact support $S \subset M$, where $S$ is  smooth enough and $u$ is of class
$C^2$  having only non-degenerate critical points in the interior of $S$.
Based on classical Morse theory and density arguments (see Aubin \cite{Aubin}), it is enough to consider test functions $u \in \mathscr{C}(M)$ in order to handle generic Sobolev inequalities.

%In the sequel let us consider a reversible Minkowski space $(\mathbb{R}^n, H)$, endowed with its canonical volume form $\mathrm{d} v_H$, such that $W_H(1) = \{ x \in \mathbb{R}^n : H(x) < 1 \}$ has measure ${\rm Vol}(W_H(1)) = \omega_n$. 

The anisotropic rearrangement of a bounded set $\Omega \subset M$ w.r.t.\ the normalized Minkowski norm $H$ is a Wulff-shape 
\begin{equation} \label{rearrangement_of_a_set}
\Omega_H^\star \coloneqq W_H(R),
\end{equation}
where $R>0$ is chosen such that $\mathrm{Vol}_F(\Omega) = \mathrm{Vol}_H(\Omega_H^\star) = \mathrm{Vol}(\Omega_H^\star)$.

Similarly to Druet, Hebey and Vaugon \cite{Druetetal}, and Alvino, Ferone, Trombetti and Lions \cite{AIHP-Lions}, for every function $u: M \to [0,\infty)$ belonging to  $\mathscr{C}(M)$, one can associate its anisotropic rearrangement $u_H^\star: \mathbb{R}^n \to [0,\infty)$, defined as
\begin{equation}\label{rearrangement_def}
u_H^\star(x) \coloneqq v(\omega_n H(x)^n), \quad \forall x \in \mathbb{R}^n
\end{equation}
for some nonincreasing function $v:[0, \infty) \to [0, \infty)$, such that for every $t \geq 0$, one has
\begin{equation}\label{vol-egyenloseg}
\mathrm{Vol}_F(\{x \in M: u(x) > t\}) = \mathrm{Vol}_H(\{x \in \mathbb{R}^n: u_H^\star(x) > t\}).
\end{equation}
Note that by definition, we have 
\begin{equation}\label{volume-preservation}
\mathrm{Vol}_F({\rm supp}(u)) = \mathrm{Vol}_H({\rm supp}(u_H^\star)),
\end{equation}
while by the layer cake representation (see Lieb and Loss \cite[Theorem 1.13]{LL}), it follows that
\begin{equation}\label{norm-preservation}
\|u\|_{L^q(M)}=\|u_H^\star\|_{L^q(\mathbb{R}^n)},\ \ \ \forall q\in (0,\infty].
\end{equation}
Furthermore, by using the isoperimetric inequality \eqref{isoperimetric}, we can prove the following P\'olya-Szeg\H{o}-type result, which represents a crucial ingredient in our arguments: 

\begin{theorem}\label{th:Polya-Szego}
Let $(M, F)$ be a noncompact, complete $n$-dimensional reversible Finsler manifold with $\mathsf{Ric}_n \geq 0$ and $0 < \mathsf{AVR}_F \leq 1$, and let $H:\mathbb R^n\to [0, \infty)$ be a normalized Minkowski norm.  Then, for every $u \in \mathscr{C}(M)$ and $p>1$, we have 	
\begin{equation}\label{Polya-Szego}
\int_M F^*(x,D u(x))^p \mathrm{d}v_F  \geq \mathsf{AVR}_F^\frac{p}{n}
\int_{\mathbb{R}^n} H^*(D u_H^\star(x))^p  \mathrm{d}v_H.
\end{equation}
\end{theorem}
\begin{proof}
We adapt the arguments of Hebey \cite{Hebey} to the anisotropic setting. 
Let $0 < t < \max_M u$ be arbitrarily fixed, and consider the sets
$$\Omega_t = \{x \in M: u(x) > t\} \subset M ~ \text{ and } ~ \Omega^\star_t = \{x \in \mathbb{R}^n: u_H^\star(x) > t\} \subset \mathbb{R}^n,$$
and the level sets
$$\Gamma_t = u^{-1}(t) ~ \text{ and } ~ \Gamma^\star_t = (u_H^\star)^{-1}(t),$$
which turn out to be the boundaries of $\Omega_t$ and $\Omega^\star_t$, respectively.

By definition of $u_H^\star$, it follows that the set $\Omega^\star_t$ has a Wulff-shape such that
\begin{equation}\label{volumes}
\text{Vol}_F(\Omega_t) = \text{Vol}_H(\Omega^\star_t) \eqqcolon  \mathcal{V}(t) .
\end{equation}
 Then, the isoperimetric inequality \eqref{isoperimetric} and the equality \eqref{isoperimetric_equality} in case of  Wulff-shapes on a Minkowski space implies that
\begin{align}\label{isoperimetic_equality}
\mathcal{P}_F(\Gamma_t) & \geq n \omega_n^{\frac{1}{n}} \mathsf{AVR}_F^{\frac{1}{n}} \text{Vol}_F(\Omega_t)^{\frac{n-1}{n}} \nonumber  = n \omega_n^{\frac{1}{n}} \mathsf{AVR}_F^{\frac{1}{n}} \text{Vol}_H(\Omega^\star_t)^{\frac{n-1}{n}}  \nonumber \\
& = \mathsf{AVR}_F^{\frac{1}{n}} \cdot \mathcal{P}_H(\Gamma^\star_t).
\end{align} 
By using relation \eqref{volumes} and the co-area formula proved by Shen \cite[Theorem 3.3.1, p.\ 46]{Shen01}, it follows that
\begin{equation*}
\mathcal{V}(t) = \int_t^\infty \left( \int_{\Gamma_s} \frac{1}{F^*(x, D u(x))} \mathrm{d} \sigma_F \right) \mathrm{d} s =  \int_t^\infty \left( \int_{\Gamma^\star_s} \frac{1}{H^*(D u_H^\star(x))} \mathrm{d} \sigma_H \right) \mathrm{d} s,
\end{equation*} 	
where $\mathrm{d} \sigma_F$ and $\mathrm{d} \sigma_H$ denote the $(n-1)$-dimensional Hausdorff measures induced by $\mathrm{d}v_F$ and $\mathrm{d}v_H$, respectively. 
It follows that
\begin{equation}\label{V1}
\mathcal{V}'(t) = - \int_{\Gamma_t} \frac{1}{F^*(x,D u(x))} \mathrm{d} \sigma_F  = - \int_{\Gamma^\star_t} \frac{1}{H^*(D u_H^\star(x))} \mathrm{d} \sigma_H.
\end{equation} 

On the one hand, since $u_H^\star$ is anisotropically symmetric in $H(x)$, the quantity $H^*(D u_H^\star)$ is constant on $\Gamma^\star_t$, thus
\begin{equation}\label{V2}
\mathcal{V}'(t) = - \frac{\mathcal{P}_H(\Gamma^\star_t)}{H^*(D u_H^\star(x))}  \quad \text{for every}  ~ x \in \Gamma^\star_t.
\end{equation}

On the other hand, applying Hölder's inequality and using relation \eqref{V1}, for every $p>1$, we have
\begin{align*}
\mathcal{P}_F(\Gamma_t) & = \int_{\Gamma_t} \mathrm{d} \sigma_F = \int_{\Gamma_t} \frac{1}{F^*(x, D u(x))^\frac{p-1}{p}} F^*(x, D u(x))^\frac{p-1}{p} \mathrm{d} \sigma_F \\
& \leq \left(-\mathcal{V}'(t) \right)^\frac{p-1}{p} \left( \int_{\Gamma_t} F^*(x, D u(x))^{p-1} \mathrm{d} \sigma_F \right)^\frac{1}{p}.
\end{align*}
Accordingly, relations \eqref{isoperimetic_equality} and \eqref{V2} yield that
\begin{align*}
\int_{\Gamma_t} F^*(x, D u(x))^{p-1} \mathrm{d} \sigma_F & \geq 
\mathcal{P}_F(\Gamma_t)^p \left(-\mathcal{V}'(t) \right)^{1-p} \geq \mathsf{AVR}_F^{\frac{p}{n}} \cdot 
\mathcal{P}_H(\Gamma^\star_t)^p \left( \frac{\mathcal{P}_H(\Gamma^\star_t)}{H^*(D u_H^\star(x))} \right)^{1-p} \\
&= \mathsf{AVR}_F^{\frac{p}{n}} 
\int_{\Gamma^\star_t} H^*(D u_H^\star(x))^{p-1} \mathrm{d} \sigma_H. 
\end{align*} 
Applying the co-area formula once again, we obtain
\begin{align*}
\int_M F^*(x, D u(x))^p \mathrm{d}v_F & = \int_0^\infty \left( \int_{\Gamma_t} F^*(x, D u(x))^{p-1} \mathrm{d} \sigma_F \right) \mathrm{d}t \\ 
& \geq \mathsf{AVR}_F^{\frac{p}{n}}  
\int_0^\infty \left( \int_{\Gamma^\star_t} H^*(D u_H^\star(x))^{p-1} \mathrm{d} \sigma_H \right) \mathrm{d}t   \\ 
& = \mathsf{AVR}_F^{\frac{p}{n}} 
\int_{\mathbb{R}^n} H^*(D u_H^\star(x))^p \mathrm{d}v_H,
\end{align*} 
which concludes the proof.
\end{proof}

A Hardy-Littlewood-P\'olya-type argument (see Lieb and Loss \cite[Theorem 3.4]{LL}) combined with a careful application of the Bishop-Gromov comparison principle yields the following rearrangement inequality:

\begin{proposition} \label{Hardy_Lemma-0} {\it 
	Let $(M, F)$ be a  complete $n$-dimensional reversible Finsler manifold with $\mathsf{Ric}_n \geq 0$, $x_0 \in M$ be any fixed point and  $H:\mathbb R^n\to [0, \infty)$ be a normalized Minkowski norm. Let $p>1$. Then for every decreasing function $f: [0, \infty) \to [0, \infty)$, one has 
	\begin{equation} \label{rearrangement_ineq}
	\int_M u(x)^p f(d_F(x_0,x)) {\rm d}v_F \leq \int_{\mathbb{R}^n} {u_H^\star}(x)^p f(H(x)) {\rm d}v_H, \ \ \ \forall u \in \mathscr{C}(M).
	\end{equation}}
\end{proposition}
	
	\begin{proof}
		Let us denote the distance function from the point $x_0 \in M$ by $d_{x_0}(x) \coloneqq d_F(x_0,x), \forall x \in M$.	
		By Fubini's theorem, \eqref{rearrangement_ineq} is equivalent to
		\begin{equation*}
		\int_0^\infty \int_0^\infty \int_M  \chi_{\{u^p>t\}}(x) \chi_{\{f \circ d_{x_0} > s\}}(x)  {\rm d}v_F \, {\rm d}t \, {\rm d}s \leq 
		\int_0^\infty \int_0^\infty \int_{\mathbb{R}^n}  \chi_{\{{u_H^\star}^p>t\}}(x) \chi_{\{f \circ H > s\}}(x)  {\rm d}v_H \, {\rm d}t \, {\rm d}s,
		\end{equation*}
		where $\chi_S$ denotes the characteristic function of a  set $S\neq \emptyset$.  
		
		Now let $t$ and $s \in [0, \infty)$ be arbitrarily fixed, and define the set \mbox{$S_{t, \rho} \coloneqq \big\{ x \in M: \chi_{\{u^p>t\}}(x) > \rho  \big\}$} 
		for every $\rho \geq 0$. By using the fact that $f$ is decreasing,  the layer cake representation implies that 
		\begin{align*}
		I_{t,s}(u,f) & \coloneqq \int_M  \chi_{\{u^p>t\}}(x) \chi_{\{f \circ d_{x_0} > s\}}(x)  {\rm d}v_F \\ 
		& = \int_{B_{x_0}(f^{-1}(s))}  \chi_{\{u^p>t\}}(x)  {\rm d}v_F \\
		& = \int_0^\infty {\rm Vol_F} \big\{x \in B_{x_0}(f^{-1}(s)):  \chi_{\{u^p>t\}}(x) > \rho\big\}  {\rm d}\rho  \\
		& = \int_0^\infty {\rm Vol_F} \big( B_{x_0}(f^{-1}(s)) \cap S_{t, \rho} \big) {\rm d}\rho .
		\end{align*}
		As $(M,F)$ has nonnegative $n$-Ricci curvature, by the Bishop-Gromov comparison principle (see Shen \cite{Shen-volume})  we have that ${\rm Vol}_F(B_{x_0}(R)) \leq {\rm Vol}(B_0(R)) = {\rm Vol}(W_H(R))$ for every $R>0$,  where $B_0(R) \subset \mathbb R^n$ denotes the Euclidean ball with center in the origin and radius $R>0$, while $W_H(R) \subset \mathbb R^n$ has a  Wulff-shape defined in \eqref{Wulff-shape}.
		Therefore, using relation \eqref{vol-egyenloseg} and a layer cake representation again, we obtain
		\begin{align*}
		I_{t,s}(u,f) & \leq  \int_0^\infty \min \Big\{ {\rm Vol_F} \big( B_{x_0}(f^{-1}(s))\big), {\rm Vol_F} \big( S_{t, \rho} \big) \Big\}  {\rm d}\rho \\ 
		& \leq \int_0^\infty \min \Big\{ {\rm Vol} \big( B_{0}(f^{-1}(s))\big), {\rm Vol} \big( S^\star_{t, \rho} \big) \Big\}  {\rm d}\rho   = \int_0^\infty {\rm Vol} \Big( W_H(f^{-1}(s)) \cap  S^\star_{t, \rho} \Big) {\rm d}\rho  \\
		& = \int_{\mathbb{R}^n} \chi_{\{f \circ H > s\}}(x)  \chi_{\{{u_H^\star}^p>t\}}(x) {\rm d}v_H,
		\end{align*}
		where $S_{t, \rho}^\star$ denotes the anisotropic rearrangement of $S_{t, \rho}$ w.r.t the norm $H$, in the sense of \eqref{rearrangement_of_a_set}.
		This completes the proof.
	\end{proof}

In particular, Proposition \ref{Hardy_Lemma-0} yields the following Hardy-type rearrangement inequality: 
\begin{equation} \label{Hardy_rearrangement}
\int_M \frac{u(x)^p}{d_F(x_0,x)^p} {\rm d}v_F \leq \int_{\mathbb{R}^n} \frac{u_H^\star(x)^p}{H(x)^p} {\rm d}v_H, \quad \forall u \in \mathscr{C}(M).
\end{equation}

\begin{remark}
In fact, for a fixed test function $u \in \mathscr{C}(M)$, we can consider multiple rearrangements, which are all equimeasurable in the following sense.
Let $H_1$ and $H_2$ be two reversible normalized Minkowski norms on $\mathbb{R}^n$.
% with their induced volume forms $\mathrm{d} v_{H_1}$ and $\mathrm{d} v_{H_2}$, such that their respective Wulff sets (see \eqref{Wulff-shape}) have measures ${\rm Vol}(W_{H_1}(1)) = {\rm Vol}(W_{H_2}(1)) = \omega_n$. 
 We can associate to $u$ its anisotropic rearrangements w.r.t. both $H_1$ and $H_2$, in the form of $u_{H_1}^\star, u_{H_2}^\star: \mathbb{R}^n \to [0,\infty)$. Then, we have
\begin{equation}\label{Equimeasurable_Dirichlet_norm}
\int_{\mathbb{R}^n} H_1^*(D u^\star_{H_1}(x))^q \mathrm{d}v_{H_1} = \int_{\mathbb{R}^n} H_2^*(D u^\star_{H_2}(x))^q \mathrm{d}v_{H_2},
\end{equation}
for any $q \in (0, \infty)$.
Indeed, by using definition \eqref{rearrangement_def}, there exists a nonincreasing function $g: [0, \infty) \to [0, \infty)$ such that $u^\star_{H_1}(x) = g(H_1(x))$ and $u^\star_{H_2}(x) = g(H_2(x))$. As the Finsler structures $H_1$ and $H_2$ are absolutely homogeneous, we have
$$H_i^*(Du_{H_i}^\star(x)) = H_i^*\big(g'(H_i(x))DH_i(x)\big) = |g'(H_i(x))| H_i^*(DH_i(x)), \text{ for } i=1,2.$$
Applying the eikonal equation $H_i^*(DH_i(x)) = 1$ for every $ x \in \mathbb{R}^n \setminus \{0\}$, then using a change of variables, it follows that for every $q>0$, we have 
\begin{equation*}
\int_{\mathbb{R}^n} H_i^*(D u_{H_i}^\star(x))^q \mathrm{d}v_{H_i} 
= \int_{\mathbb{R}^n} |g'(H_i(x))|^q \mathrm{d}v_{H_i}  \\ 
= n \omega_n \int_0^\infty |g'(\rho)|^q \rho^{n-1} \mathrm{d}\rho,  
\end{equation*}
where we used the fact that both Minkowski norms $H_i$ are normalized, $i=1,2$.
  	
In the case when $H(x) = |x|$ is the standard Euclidean norm, $u^\star_{|\cdot|}$ turns out to be the usual radially symmetric rearrangement of $u$, for which the sets $\{x \in \mathbb{R}^n: u^\star_{|\cdot|}(x) > t\}$ are Euclidean open balls with center $0 \in \mathbb{R}^n$. 	
This radial symmetrization has proven to be particularly useful when showing the Finslerian counterparts of several Euclidean Sobolev-type inequalities, see \S \ref{sec:Sobolev_inequalities}.

On the other hand, the anisotropic rearrangement \eqref{rearrangement_def} w.r.t. an arbitrary reversible Minkowski norm $H$ can be a key technique when considering rigidity results in the spirit of Balogh and Krist\'aly \cite{BaloghKristaly} and Krist\'aly \cite{Kristaly-JMPA}.
More specifically, we believe that in certain sharp Sobolev inequalities, equality holds for some nonzero extremal function $u \in \mathscr{C}(M)$ if and only if $\mathsf{AVR}_F = 1$; in particular, when $(M,F)$ is a Finsler manifold of Berwald type, this would imply that $(M,F)$ is isometric to a Minkowski space $(\mathbb{R}^n, H)$.
\end{remark}

Finally, we present the following auxiliary result, whose proof, similarly to Balogh and Krist\'aly \cite[Lemma 3.1]{BaloghKristaly}, follows in a straightforward way by the layer cake representation: 

\begin{lemma}\label{lemma-layer} {\rm (see \cite{BaloghKristaly})}
Let $(M,F)$ be a complete, reversible Finsler manifold, $R>0$ and $x_0\in M$ an arbitrarily fixed point. If $f:[0,R] \to \mathbb R$ is a $C^1$-function on $(0,R)$, then $$\int_{B_{x_0}(R)} f(d_F(x_0,x)){\rm d}v_F = f(R){\rm Vol}_F(B_{x_0}(R)) - \int_0^R f'(r){\rm Vol}_F(B_{x_0}(r)){\rm d}r.$$
\end{lemma}

\vspace{0.1cm}

\section{Morrey-Sobolev and Hardy-Sobolev Inequalities}  \label{sec:Sobolev_inequalities}

\subsection{Morrey-Sobolev interpolation inequality: sharp support-bound.}

Let $p > n \geq 2$. In the Euclidean case, Talenti \cite[Theorem 2.E]{Talenti} proved the following Morrey-Sobolev inequality 
\begin{equation}\label{Talenti-2}
	\| u\|_{L^\infty(\mathbb R^n)} \leq{\sf
		T}_{p,n}\, {\rm Vol}({\rm supp}\, u)^{\frac{1}{n}-\frac{1}{p}} \| \nabla u\|_{L^p(\mathbb R^n)} ,\ \ \ \forall u\in C_0^\infty(\mathbb R^n),
\end{equation}
where ${\rm supp}\, u\subset \mathbb R^n$ denotes the support of $u$, and $\mathrm{Vol}({\rm supp}\, u)$ stands for the Euclidean volume of the set ${\rm supp} u$; moreover, the constant 
\begin{equation}\label{constant-morrey-2}
	{\sf
		T}_{p,n}=n^{-\frac{1}{p}}\omega_n^{-\frac{1}{n}}\left(\frac{p-1}{p-n}\right)^\frac{1}{p'}
\end{equation}
is sharp and achieved by the function 
$$u(x)=
\left(1-|x|^\frac{p-n}{p-1}\right)_+.$$
The counterpart of \eqref{Talenti-2} on Finsler manifolds with nonnegative $n$-Ricci-curvature is given by Theorem  \ref{theorem-1-intro}.  \\

%\label{Th_Morrey2} 	{\rm (=\textbf{Theorem}\  \ref{theorem-1-intro})}	
%	Let $(M, F)$ be a noncompact, complete $n$-dimensional reversible Finsler manifold with ${\sf Ric}_n\geq 0$ and $0 < \mathsf{AVR}_F \leq 1$, and for every $p>n\geq 2$ the constant $\mathsf{T}_{p,n}$ given by \eqref{constant-morrey-2}. Then, for every  $u \in C_0^\infty(M)$ one has 
%	\begin{equation}\label{Morrey-Sobolev-Finsler-intro}
%		\| u\|_{L^\infty(M)} \leq{\sf T}_{F}^{\sf MS} \, {\rm Vol}_{F}({\rm supp}\, u)^{\frac{1}{n}-\frac{1}{p}}\left(\int_{M} F^*(x, D u(x))^p \mathrm{d}v_F \right)^{\frac{1}{p}}  ,
%	\end{equation}
%	where the constant 
%	$ {\sf T}_{F}^{\sf MS}={\sf
%		T}_{p,n}\, {\sf AVR}_F^{-\frac{1}{n}}$ is sharp.
	
	{\it Proof of Theorem  \ref{theorem-1-intro}.}
		Let $u \in C_0^\infty(M)$ be arbitrarily fixed, and consider its spherically symmetric rearrangement $u_{|\cdot|}^\star: \mathbb{R}^n \to [0, \infty)$ w.r.t. the Euclidean norm $|\cdot|$.
%		
%		Similarly to Theorem \ref{Th_Morrey1}, inequality \eqref{Morrey-Sobolev-Finsler-intro} follows by using symmetrization arguments. Let $u \in C_0^\infty(M)$ be arbitrarily fixed, and consider its spherically symmetric rearrangement $u_{|\cdot|}^\star: \mathbb{R}^n \to [0, \infty)$ with respect to the Euclidean norm $|\cdot|$.
		Then, by relations \eqref{volume-preservation}-\eqref{Polya-Szego} and \eqref{Talenti-2}, it follows that
		\begin{align*}
			\| u\|_{L^\infty(M)} = \| u_{|\cdot|}^\star\|_{L^\infty(\mathbb{R}^n)} 
			& \leq {\sf T}_{p,n}\,  {\rm Vol}({\rm supp}\, u_{|\cdot|}^\star)^{\frac{1}{n}-\frac{1}{p}} \|\nabla u_{|\cdot|}^\star \|_{L^p(\mathbb R^n)} \\
			& \leq 
			{\sf T}_{p,n} \, {\sf AVR}_F^{-\frac{1}{n}} {\rm Vol}_{F}({\rm supp}\, u)^{\frac{1}{n}-\frac{1}{p}} \left(\int_{M} F^*(x, D u(x))^p \mathrm{d}v_F \right)^{\frac{1}{p}},
		\end{align*}
		which is precisely \eqref{Morrey-Sobolev-Finsler-intro}. 
		
We assume by contradiction that there exists a constant $\mathcal C< {\sf T}_{F}^{\sf MS}={\sf T}_{p,n} \, {\sf AVR}_F^{-\frac{1}{n}} $ such that 
		\begin{equation}\label{egyenlet-7}
			\| u\|_{L^\infty(M)} \leq\mathcal C \, \, {\rm Vol}_{F}({\rm supp}\, u)^{\frac{1}{n}-\frac{1}{p}}\left(\int_{M} F^*(x, D u(x))^p \mathrm{d}v_F \right)^{1/p}  ,
			\ \ 
			\forall u\in C_0^\infty(M).
		\end{equation}
		We fix $x_0\in M$ and $R>0$, and define the function  $u_R: M \to [0, \infty)$  by
		$$u_R(x)=\left(1-\left(\frac{d_F(x_0,x)}{R}\right)^\frac{p-n}{p-1}\right)_+,\ x\in M.
		$$
		By the eikonal equation \eqref{tavolsag-derivalt} and Lemma \ref{lemma-layer}, it follows that
		\begin{eqnarray*}
			\int_{M} F^*(x, D u_R(x))^p \mathrm{d}v_F
			&=&\frac{1}{R^p} \left(\frac{p-n}{p-1}\right)^p	\int_{B_{x_0}(R)}\left(\frac{d_F(x_0,x)}{R}\right)^{(1-n)p'}{\rm d}v_F \\ 
			&=& \frac{1}{R^p} \left(\frac{p-n}{p-1}\right)^p \left. \bigg({\rm Vol}_F(B_{x_0}(R))-\right.\\&& \left.-(1-n)p'\int_0^1t^{(1-n)p'-1}{\rm Vol}_F(B_{x_0}(Rt)){\rm d}t\right).
		\end{eqnarray*}
		The latter relation and Lebesgue's dominated convergence theorem imply that 
		$$\lim_{R\to \infty}\frac{1}{R^{n-p}} \int_{M} F^*(x, D u_R(x))^p \mathrm{d}v_F 
		= n \omega_n \left(\frac{p-n}{p-1}\right)^{p-1}{\sf AVR}_F.$$
		Note that $\| u_R\|_{L^\infty(M)}=1$ and ${\rm supp}\, u_R=B_{x_0}(R)$. If we use $u_R$ as a test function in \eqref{egyenlet-7}, an argument via  
		the latter limit gives that
		$$1\leq \mathcal C\,  \omega_n^\frac{1}{n}n^\frac{1}{p}\left(\frac{p-n}{p-1}\right)^\frac{1}{p'}{\sf AVR}_F^\frac{1}{n},$$
		which is equivalent to $\mathcal C\geq {\sf
			T}_{p,n}\, {\sf AVR}_F^{-\frac{1}{n}}={\sf T}_{F}^{\sf MS},$ contradicting our initial assumption.	\hfill $\square$

\subsection{Morrey-Sobolev interpolation inequality: sharp $L^1$-bound.}
The following Morrey-Sobolev inequality is proved by Talenti \cite[Theorem 2.C]{Talenti}, stating that for every $p > n \geq 2$, one has
\begin{equation}\label{Morrey_Sobolev_Talenti}
\|u\|_{L^\infty(\mathbb{R}^n)} \leq
\mathsf{C}_{p,n} \|u\|^{1-\eta}_{L^1(\mathbb{R}^n)} \|\nabla u\|^{\eta}_{L^p(\mathbb{R}^n)}, \ \ \ \forall u \in C_0^\infty(\mathbb{R}^n),
\end{equation}
where 
\begin{equation}\label{eta-def}
\eta = \frac{np}{np+p-n}.
\end{equation}
Moreover, the constant 
\begin{equation}\label{constant-morrey-1}
\mathsf{C}_{p,n}  =  (n\omega_n^\frac{1}{n})^{-\frac{np'}{n+p'}}\left(\frac{1}{n}+\frac{1}{p'}\right)
\left(\frac{1}{n}-\frac{1}{p}\right)^\frac{(n-1)p'-n}{n+p'}
\left({\sf B}\left(
\frac{1-n}{n}p'+1,p'+1\right)\right)^\frac{n}{n+p'}
\end{equation}
is sharp and achieved by the function 
$$u(x)=\left\{
\begin{array}{lll}
\displaystyle\int_{|x|}^1
r^\frac{1-n}{p-1}(1-r^n)^\frac{1}{p-1}{\rm d}r, & & {\rm
	if}\ |x|\leq 1;
\\ 0,&  & {\rm otherwise}.
\end{array}
\right.$$
%Here $p' = \frac{p}{p-1}$, and $\sf B(\cdot, \cdot)$ denotes the Euler beta-function.

We prove the counterpart of  \eqref{Morrey_Sobolev_Talenti} on Finsler manifolds with nonnegative $n$-Ricci-curvature:\\

%
%\begin{theorem} {\rm (=\textbf{Theorem}\  \ref{Th_Morrey1-0})}	  
%	 \label{Th_Morrey1}
%Let $(M, F)$ be a noncompact, complete $n$-dimensional reversible Finsler manifold with  ${\sf Ric}_n\geq 0$ and $0 < \mathsf{AVR}_F \leq 1$, and for any $p>n\geq 2$ the constants $\eta$ and $\mathsf{C}_{p,n}$ given by \eqref{eta-def} and \eqref{constant-morrey-1}, respectively. Then, for every $u \in C_0^\infty(M)$ one has 
%\begin{equation}\label{Morrey_Sobolev_Finsler-0}
%\|u\|_{L^\infty(M)} \leq
%\mathsf{C}_{F}^{\sf MS} \left(\int_M |u(x)| \mathrm{d}v_F \right)^{1-\eta}
%\left(\int_M F^*(x, D u(x))^p \mathrm{d}v_F \right) ^{\frac{\eta}{p}}, 
%\end{equation}
%where the constant 
%$
%\mathsf{C}_{F}^{\sf MS}  =  \mathsf{C}_{p,n} {\sf AVR}_F^{-\frac{\eta}{n}}
%$
%is sharp. 
%\end{theorem}

{\it Proof of Theorem   \ref{Th_Morrey1-0}.}
Let $u \in C_0^\infty(M)$ be arbitrarily fixed, and consider its spherically symmetric rearrangement $u_{|\cdot|}^\star: \mathbb{R}^n \to [0, \infty)$ w.r.t. the Euclidean norm $|\cdot|$. By applying relations \eqref{norm-preservation} and \eqref{Polya-Szego}, and using Talenti's inequality \eqref{Morrey_Sobolev_Talenti}, it follows that	
\begin{align}\label{eq:chain_ineq}
\|u\|_{L^\infty(M)} = \|u^\star_{|\cdot|}\|_{L^\infty(\mathbb{R}^n)} 
& \leq 
\mathsf{C}_{p,n} \|u^\star_{|\cdot|}\|^{1-\eta}_{L^1(\mathbb{R}^n)} \|\nabla u^\star_{|\cdot|}\|^{\eta}_{L^p(\mathbb{R}^n)} \nonumber \\
& \leq \mathsf{C}_{p,n} \|u\|^{1-\eta}_{L^1(M)} \mathsf{AVR}_F^{-\frac{^\eta}{n}} \left(\int_{M} F^*(x, D u(x))^p \mathrm{d}v_F \right)^{\frac{\eta}{p}}, 
\end{align}
which is exactly \eqref{Morrey_Sobolev_Finsler-0}. 

As for the optimality  of the constant, we assume that $\mathsf{C}_{F}^{\sf MS}  =  \mathsf{C}_{p,n} {\sf AVR}_F^{-\frac{\eta}{n}}$ is not sharp in \eqref{Morrey_Sobolev_Finsler-0}, i.e., there exists $\mathcal C < \mathsf{C}_{F}^{\sf MS}$ such that
\begin{equation}\label{eq:not_sharp1}
\|u\|_{L^\infty(M)} \leq
\mathcal C\, \left(\int_M |u(x)| \mathrm{d}v_F \right)^{1-\eta}
\left(\int_M F^*(x, D u(x))^p \mathrm{d}v_F \right) ^{\frac{\eta}{p}}, \ \forall u\in C_0^\infty(M).
\end{equation}
Let $h,H:(0,1]\to \mathbb R$ be the functions
$$h(r)=r^\frac{1-n}{p-1}(1-r^n)^\frac{1}{p-1} \quad {\rm and } \quad  H(s)=\int_0^sh(r){\rm d}r.$$ 
Let $x_0\in M$ and $R>0$ be fixed, and consider the function  $u_R: M \to [0, \infty)$ defined by
$$u_R(x)= \left\{
\begin{array}{lll}
\displaystyle H(1)-H\left(\frac{d_F(x_0,x)}{R}\right), & & {\rm if}\ x\in B_{x_0}(R);
\\ 0,&  & {\rm otherwise}.
\end{array}
\right. $$
First, we have that 
\begin{equation}\label{vegtelen-estimate}
\| u_R\|_{L^\infty(M)} =H(1)=\int_0^1h(r){\rm d}r=\frac{1}{n}{\sf B}\left(
\frac{1-n}{n}p'+1,p'\right).
\end{equation}
By using Lemma \ref{lemma-layer} and a change of variables, it turns out that 
\begin{eqnarray}\label{11-00}
\nonumber	\int_M |u_R(x)| \mathrm{d}v_F &=& \int_{B_{x_0}(R)}\left(H(1)-H\left(\frac{d_F(x_0,x)}{R}\right)\right){\rm d}v_F\\&=& \frac{1}{R}\int_0^R{\rm Vol}_F(B_{x_0}(r))H'\left(\frac{r}{R}\right){\rm d}r\nonumber 
\\&=& \int_0^1{\rm Vol}_F(B_{x_0}(Rt))h(t){\rm d}t.
\end{eqnarray}
On the other hand, since 
$$D u_R(x) = -\frac{1}{R} H'\left(\frac{d_F(x_0,x)}{R}\right) D d_F(x_0,x)\ \ {\rm for\ a.e.}\ x\in B_{x_0}(R),$$
the eikonal equation \eqref{tavolsag-derivalt} and the absolute homogeneity of the Finsler structure $F$ yield that
\begin{eqnarray}\label{22-00}
\nonumber\int_M F^*(x, D u_R(x))^p \mathrm{d}v_F &=&
\frac{1}{R^p} \int_{B_{x_0}(R)}h^p\left(\frac{d_F(x_0,x)}{R}\right){\rm d}v_F\\&=& -\frac{1}{R^p}\int_0^1{\rm Vol}_F(B_{x_0}(Rt))\cdot(h^p)'(t){\rm d}t,
\end{eqnarray}
where we used Lemma \ref{lemma-layer} and a change of variables.
 
By density reasons, the function $u_R$ can be used as a test function in \eqref{eq:not_sharp1}, i.e.,
\begin{equation*}
\| u_R\|_{L^\infty(M)} \leq 
\mathcal C\, \left(\int_M |u_R(x)| \mathrm{d}v_F \right)^{1-\eta}
\left(\int_M F^*(x, D u_R(x))^p \mathrm{d}v_F \right) ^{\frac{\eta}{p}}.
\end{equation*}
Furthermore, Lebesgue's dominated convergence theorem and relations  \eqref{11-00} and \eqref{22-00} imply that
\begin{equation*}
\lim_{R\to \infty} \frac{1}{R^n}\int_M |u_R(x)| \mathrm{d}v_F = \omega_n {\sf AVR}_F \int_0^1t^nh(t){\rm d}t = \omega_n{\sf AVR}_F\, \frac{1}{n}{\sf B}\left(
\frac{1-n}{n}p'+2,p'\right),
\end{equation*}
and
\begin{align*}
\lim_{R\to \infty} \frac{1}{ R^{n-p}} \int_M F^*(x, D u_R(x))^p \mathrm{d}v_F &= -\omega_n{\sf AVR}_F \int_0^1t^n(h^p)'(t){\rm d}t \\ &= \omega_n{\sf AVR}_F\,{\sf B}\left(
\frac{1-n}{n}p'+1,p'+1\right).
\end{align*}
Therefore, by using the latter limits and relations \eqref{vegtelen-estimate} and \eqref{constant-morrey-1}, a straightforward manipulation of the above terms implies   
$${\sf C}_{p,n}\leq \mathcal C\, {\sf AVR}_F^{1-\eta+\frac{\eta}{p}}.$$
Since $$1-\eta+\frac{\eta}{p}=\frac{\eta}{n},$$ see \eqref{eta-def}, the latter inequality contradicts our initial assumption $\mathcal C<\mathsf{C}_{F}^{\sf MS}$, which yields the sharpness of $\mathsf{C}_{F}^{\sf MS}$ in \eqref{Morrey_Sobolev_Finsler-0}.  \hfill $\square$\\

Concerning the equality in the Morrey-Sobolev inequalities  \eqref{Morrey-Sobolev-Finsler-intro} and  \eqref{Morrey_Sobolev_Finsler-0}, we can state the following rigidity result in the  case when $(M,F) = (M,g)$ is a Riemannian manifold; in particular, it turns out that the existence of nonzero extremal functions implies that the manifold is isometric to the Euclidean space:

\begin{theorem} \label{Th_Morrey_Rigidity}
	Let $(M, g)$ be a noncompact, complete $n$-dimensional Riemannian manifold having $\mathsf{Ric} \geq 0$, $0 < \mathsf{AVR}_g \leq 1$, and let $2 \leq n < p$. Then the following statements are equivalent: 
	\begin{enumerate}
		\item[{\rm (i)}]  Equality holds in 
		\eqref{Morrey-Sobolev-Finsler-intro}
		 for some nonzero and nonnegative function $u \in  \mathscr{C}(M);$
		\item[{\rm (ii)}]  Equality holds in  \eqref{Morrey_Sobolev_Finsler-0} for some nonzero and nonnegative function $u \in  \mathscr{C}(M);$
		\item[{\rm (iii)}]  $(M, g)$ is isometric to the Euclidean space $(\mathbb R^n, g_0)$.
	\end{enumerate}

	\begin{proof} The proofs of the equivalences 
		 ${({\rm i}) \Leftrightarrow ({\rm iii})}$ and ${({\rm ii}) \Leftrightarrow ({\rm iii})}$ are  analogous; we shall present the latter. 		Suppose that equality holds in \eqref{Morrey_Sobolev_Finsler-0} for some nonzero and nonnegative function $u \in  \mathscr{C}(M)$. Consequently, equalities hold in the chain of inequalities \eqref{eq:chain_ineq}, therefore we have equality in the P\'olya-Szeg\H o inequality \eqref{Polya-Szego} as well. As the latter inequality is rigid in the Riemannian case, see Balogh and Krist\'aly \cite[Proposition 3.1]{BaloghKristaly}, we obtain that 
		$$\mathsf{AVR}_g = 1,$$ i.e., $(M, g)$ is isometric to the Euclidean space $(\mathbb R^n, g_0)$, see e.g. Petersen \cite{Petersen}. The converse is trivial.
	\end{proof}
\end{theorem}

\begin{remark}\rm
	A natural question arises on the validity of Theorem \ref{Th_Morrey_Rigidity} not only for functions belonging to $\mathscr{C}(M)$ but to the appropriate Sobolev spaces associated to the Morrey-Sobolev inequalities  \eqref{Morrey-Sobolev-Finsler-intro}
	and  \eqref{Morrey_Sobolev_Finsler-0}. Such a question requires a deeper analysis, since usually the 'small' subspace of functions where Sobolev inequalities can be easily obtained do not contain the expected extremal functions, while after the approximation/density arguments we cannot track back the equality cases in the proof; see e.g. Brothers and Ziemer \cite{BZ}, Balogh and Krist\'aly \cite{BaloghKristaly}. In our case (Theorem \ref{Th_Morrey_Rigidity}) however, the corresponding extremal functions belong to $\mathscr{C}(M)$. 
\end{remark}

In the spirit of the latter result, one should ask whether there exist similar rigidity statements in the general, Finslerian setting as well. 
More specifically, we formulate the following question:

\begin{quotation}
\textit{If $(M,F)$ is a noncompact, complete $n$-dimensional reversible Finsler manifold with  ${\sf Ric}_n\geq 0$, $0 < \mathsf{AVR}_F \leq 1$, and $2 \leq n < p$, is it true that if there exists a nonzero and nonnegative extremal function $u \in  \mathscr{C}(M)$ of inequality \eqref{Morrey-Sobolev-Finsler-intro} $($or \eqref{Morrey_Sobolev_Finsler-0}, respectively$),$ then $\mathsf{AVR}_F = 1$ $($thus, the manifold is a locally Minkowski space$)?$}
\end{quotation}
We believe that the anisotropic rearrangement \eqref{rearrangement_def} w.r.t.\ an arbitrary reversible Minkowski norm $H$ shall be a key ingredient when considering such problems. However,  the lack of a rigid isoperimetric inequality impedes this investigation; indeed, the characterization of the equality in the isoperimetric inequality \eqref{isoperimetric} -- and hence in the P\'olya-Szeg\H{o} inequality \eqref{Polya-Szego} -- is currently available only on Riemannian manifolds, see Brendle \cite{Brendle}, and Balogh and Krist\'aly \cite{BaloghKristaly}.

\subsection{Hardy-Sobolev-type inequalities.}

In this section we consider Sobolev inequalities involving a Hardy-type singular term of the form $x \mapsto d_F(x_0,x)^{-p}$, where $x_0\in M$ is any fixed point, $p>1$. \\

%\begin{theorem}\label{Th:Hardy}{\rm (=\textbf{Theorem}\  \ref{theorem-Hardy})}
%Let $(M, F)$ be a noncompact, complete $n$-dimensional reversible Finsler manifold with  ${\sf Ric}_n\geq 0$, $0 < \mathsf{AVR}_F \leq 1$, and $n >p>1$. Then, for every $x_0 \in M$ and $u \in C_0^\infty(M)$, the following Hardy-Sobolev inequality holds:
%\begin{equation}\label{Hardy}
%\int_{M} F^*(x, D u(x))^p \mathrm{d}v_F \geq 
%\mathsf{AVR}_F^\frac{p}{n} \left(\frac{n-p}{p}\right)^p \int_{M} \frac{|u(x)|^p}{d_F(x_0, x)^p} \mathrm{d}v_F.  
%\end{equation}
%\end{theorem}

{\it Proof of Theorem \ref{theorem-Hardy}.} Due to a density reason, since $F$ is reversible,  we may assume without loss of generality that $u\geq 0.$
%Note that for $n=2$ the statement is trivial. For $n \geq 3$, 
Let us define the symmetric rearrangement of $u$ w.r.t.\   the Euclidean norm $|\cdot|$, i.e.,   $u_{|\cdot|}^\star: \mathbb{R}^n \to [0, \infty)$.	
Using the rearrangement inequalities \eqref{Polya-Szego},  \eqref{Hardy_rearrangement}, and the classical Euclidean Hardy-Sobolev inequality, see e.g. Balinsky, Evans and
Lewis \cite[Corollary 1.2.6]{BalinskyEvansLewis}, we have
\begin{align*}
\int_{M} F^*(x, D u(x))^p \mathrm{d}v_F & \geq  \mathsf{AVR}_F^{\frac{p}{n}} \int_{\mathbb{R}^n} | \nabla u_{|\cdot|}^\star(x) |^p \mathrm{d}x 
\\ & \geq \mathsf{AVR}_F^{\frac{p}{n}} \left(\frac{n-p}{p}\right)^p
\int_{\mathbb{R}^n} \frac{u_{|\cdot|}^\star(x)^p}{|x|^p} \mathrm{d}x \\
& \geq
\mathsf{AVR}_F^\frac{p}{n}
\left(\frac{n-p}{p}\right)^p \int_{M} \frac{|u(x)|^p}{d_F(x_0, x)^p} \mathrm{d}v_F,
\end{align*} 
which concludes the proof. 
\hfill $\square$

\begin{remark}
The sharpness of the Hardy-Sobolev inequality \eqref{Hardy-intro} is open in the generic Finsler setting.
This fact can be attributed to the lack of extremal functions in the Euclidean case, thus arguments similar to Theorems \ref{theorem-1-intro} and \ref{Th_Morrey1-0} no longer yield the expected conclusion.
\end{remark}

In the particular case when $p=2$, $\Omega\subset M$ is a smooth bounded open set and $x_0 \in \Omega$, we have the following \textit{Brezis-Poincar\'e-V\'azquez} inequality: 

\begin{theorem}\label{Th:BrezisPoincareVazquez}
	Let $(M, F)$ be a noncompact, complete $n$-dimensional reversible Finsler manifold with  ${\sf Ric}_n\geq 0$, $0 < \mathsf{AVR}_F \leq 1$, and $n \geq 2$. Let $\Omega\subset M$ be a smooth, bounded open set with $x_0 \in \Omega$ arbitrarily fixed. If $\mu \in \left[0,\frac{(n-2)^2}{4}{\sf AVR}_F^\frac{2}{n}\right]$, then for every $u\in C_0^\infty(\Omega)$, we have
	\begin{equation}\label{BVP-1}
	\int_\Omega F^*(x, D u(x))^2{\rm d}v_F - \mu \int_\Omega \frac{u(x)^2}{d_F(x_0,x)^2} {\rm d}v_F \geq {\sf S}_{\mu,F}(\Omega)\, \int_\Omega {u(x)^2} {\rm d}v_F,
	\end{equation}
	where
	\begin{equation} \label{S_mu_F}
	{\sf S}_{\mu,F}(\Omega) = {\sf AVR}_F^\frac{2}{n} j_{\overline \mu}^2\, \left( \frac{\omega_n}{{\rm Vol}_F(\Omega)} \right)^\frac{2}{n} \quad \text{and} \quad\quad \overline 
	\mu = 
	 \sqrt{\frac{(n-2)^2}{4}-\mu\, {\sf AVR}_F^{-\frac{2}{n}}},
	\end{equation}
	$j_{\overline \mu}$ being the first positive zero of the Bessel function of the first kind $J_{\overline \mu}$.
	
	\begin{proof}	
	Let $B \coloneqq \Omega_{|\cdot|}^\star$ and $u_{|\cdot|}^\star: \mathbb{R}^n \to [0, \infty)$ be the symmetric rearrangements of $\Omega$ and $u$ w.r.t. the Euclidean norm $|\cdot|$, i.e., $B$ is an Euclidean open ball with center in the origin such that ${\rm Vol}(B) = {\rm Vol}_{F}(\Omega)$.
	 
	On the one hand, by inequality \eqref{Hardy-intro}, the left hand side of \eqref{BVP-1} turns out to be nonnegative whenever  $\mu\leq \frac{(n-2)^2}{4}{\sf AVR}_F^\frac{2}{n}$. On the other hand, relations \eqref{Polya-Szego}, \eqref{Hardy_rearrangement} and \eqref{norm-preservation} together with the result of Krist\'aly and Szak\'al \cite[Theorem 1.1]{Kristaly-JDE} imply that if $\mu\in \left[0,\frac{(n-2)^2}{4}{\sf AVR}_F^\frac{2}{n}\right]$,  one has
	\begin{align*}
	\int_\Omega F^*(x, Du(x))^2 {\rm d}v_F - \mu\int_\Omega \frac{u(x)^2}{d_F(x_0,x)^2} {\rm d}v_F & \geq {\sf AVR}_F^\frac{2}{n}\int_B|\nabla u_{|\cdot|}^\star(x)|^2{\rm d}x-\mu\int_B \frac{u_{|\cdot|}^\star(x)^2}{|x|^2} {\rm d}x\\&\geq {\sf AVR}_F^\frac{2}{n}\, j_{\overline \mu}^2\, \omega_n^\frac{2}{n}\, {\rm Vol}(B)^{-\frac{2}{n}} \int_B  u_{|\cdot|}^\star(x)^2{\rm d}x
	\\&={\sf AVR}_F^\frac{2}{n}\, j_{\overline \mu}^2\, \omega_n^\frac{2}{n}\, {\rm Vol}_{F}(\Omega)^{-\frac{2}{n}} \int_{\Omega} u(x)^2{\rm d}v_F,
	\end{align*}
	which ends the proof. 
	\end{proof}
\end{theorem}

	We conclude this subsection with some comments concerning the sharpness and attainability of the best constant in the Brezis-Poincar\'e-V\'azquez inequality. Namely, we have that:
\begin{enumerate}[(i)]
	\item[$\bullet$] if $(M,F)$ is isometric to a Minkowski space $(\mathbb{R}^n, H)$ (thus,  ${\sf AVR}_F = 1$ in particular), the constant ${\sf S}_{\mu,F}(\Omega)$ is sharp and attained (for sufficiently small $\mu$) if and only if the set $\Omega$ has a Wulff-shape, see Krist\'aly and Szak\'al \cite[Theorem 1.1]{Kristaly-JDE}; 
	
	\item[$\bullet$] if $(M,F) = (M,g)$ is a Riemannian manifold, the constant ${\sf S}_{0,F}(\Omega)$ is sharp
	 and it is attained whenever $(M,g)$ is isometric to the usual Euclidean space $(\mathbb R^n,g_0)$ and $\Omega\subset M$ is isometric to a ball in $\mathbb R^n$, see Balogh and Krist\'aly \cite[Theorem 3.5]{BaloghKristaly}.	
\end{enumerate}
The latter statements can be reformulated in terms of eigenvalues for a model problem. Indeed, putting ourselves into  the setting of Theorem \ref{Th:BrezisPoincareVazquez}, we consider the \textit{eigenvalue problem}
$$
\begin{cases}
-\Delta_F u(x) - \mu \displaystyle\frac{u(x)}{d_F(x_0, x)^2} = \lambda u(x), \quad x \in \Omega, \\
u \in W_{0,F}^{1,2}(\Omega).
\end{cases}
\eqno{(EP)_{\mu,\lambda}}
$$
Then, we can prove that
\begin{enumerate}[(i)]
	\item[$\bullet$] if $(M,F)$ is a Minkowski space,  $\lambda={\sf S}_{\mu,F}(\Omega)$ is the first eigenvalue  (with sufficiently small $\mu$) for the problem $(EP)_{\mu,\lambda}$ if and only if  $\Omega$ has a Wulff-shape; 
	
	\item[$\bullet$] if $(M,F) = (M,g)$ is a Riemannian manifold, 
	$\lambda={\sf S}_{0,F}(\Omega)$ is the first eigenvalue for the problem $(EP)_{0,\lambda}$ if and only if $(M,g)$ is isometric to the Euclidean space $(\mathbb R^n,g_0)$ and $\Omega\subset M$ is isometric to a ball in $\mathbb R^n.$	
\end{enumerate}
In general, however, the sharpness of ${\sf S}_{\mu, F}(\Omega)$ in \eqref{BVP-1} (and its attainability) remains an open question.

\subsection{Example.}\label{subsection-example}

 Riemannian manifolds with nonnegative Ricci curvature and positive asymptotic volume ratio are provided e.g.\ in Balogh and Krist\'aly \cite{BaloghKristaly}. In the sequel, we construct a family of non-Riemannian Finsler manifolds where our results apply, i.e.,  noncompact, complete $n$-dimensional reversible Finsler manifolds $(M,F)$ with  ${\sf Ric}_n\geq 0$ and $0 < \mathsf{AVR}_F \leq 1$. 

To do this, we endow the space $\mathbb R^{n-1}$ ($n \geq 3$) with a Riemannian
metric $g$ such that $(\mathbb R^{n-1},g)$ is complete with
nonnegative Ricci curvature and assume that the induced warped metric $\tilde g$ on $\mathbb R^{n-1}\times \mathbb R$, defined by $$\tilde g_{(x,t)}(v,w)=\sqrt{g_x(v,v)+w^2},\ \  (x,t)\in \mathbb R^{n}, (v,w)\in T_{x}\mathbb R^{n-1}\times  T_t\mathbb R,$$
has the property that $0 < \mathsf{AVR}_{\tilde g} \leq 1$; the family of such metrics is rich, see \cite{BaloghKristaly}.  

For any fixed $\varepsilon>0$, we consider 
on $\mathbb R^{n}=\mathbb R^{n-1}\times \mathbb R$ the Finsler metric
$F_\varepsilon:T\mathbb R^{n}\longrightarrow [0,\infty)$ given by
\[ F_\varepsilon((x,t),(v,w))=\sqrt{g_x(v,v)+w^2 +
	\varepsilon \sqrt{g_x(v,v)^2+w^4}},\ \  (x,t)\in \mathbb R^{n}, (v,w)\in T_{x}\mathbb R^{n-1}\times  T_t\mathbb R, \]
endowed with its natural Busemann-Hausdorff measure. 
  By Krist\'aly and Ohta \cite{Kri-Ohta}, we know that $(\mathbb R^{n},F_\varepsilon)$ is a noncompact, complete, reversible non-Riemannian Berwald space with nonnegative Ricci curvature. In particular, $(\mathbb R^{n},F_\varepsilon)$ being a Berwald space, it has vanishing mean covariation, thus $\mathsf{Ric}_n \geq 0$. By Bishop-Gromov comparison principle we clearly have that $\mathsf{AVR}_{F_\varepsilon} \leq 1$. It remains to show that $(M,F_\varepsilon)$ has Euclidean volume growth, i.e., $\mathsf{AVR}_{F_\varepsilon} >0.$ To this end, we observe that  $$\tilde g_{(x,t)}(v,w)\leq F_\varepsilon((x,t),(v,w))\leq \sqrt{1+\varepsilon}\tilde g_{(x,t)}(v,w),\ \  (x,t)\in \mathbb R^{n}, (v,w)\in T_{x}\mathbb R^{n-1}\times  T_t\mathbb R,$$
  thus the density functions from  \eqref{Hausdorff_measure} verify   $ \sigma_{F_\varepsilon}\geq \sigma_{\tilde g}$. Moreover,  based on the assumption that $0 < \mathsf{AVR}_{\tilde g} \leq 1$,   simple estimates show that  $$ \mathsf{AVR}_{F_\varepsilon}\geq  \frac{\mathsf{AVR}_{\tilde g}}{(1+\varepsilon)^\frac{n}{2}}>0,$$ which concludes our claim.

\section{Applications to PDEs}\label{section-PDE}

\subsection{Multiple solutions for a Dirichlet problem	involving the $p$-Finsler-Laplacian.}

In this section, we provide an application of Theorem \ref{theorem-1-intro} by considering the Dirichlet problem
\begin{align}\label{eq:D}
\begin{cases}
-\Delta_{F,p} u(x) = \lambda h(u(x)), \quad x \in \Omega \\
u \in W_{0,F}^{1,p}(\Omega)   ,
\end{cases}
\tag{$\mathcal{D}_\lambda$}
\end{align}
where   $(M, F)$ is an $n$-dimensional Finsler manifold, $\Omega \subset M$ is a bounded open set with $C^1$ boundary, $\Delta_{F,p}$ is the  $p$-Finsler-Laplace operator with  $p > n$, $\lambda>0,$ and $h: \mathbb{R} \to \mathbb{R}$ is a continuous function with $h(0) = 0$. Furthermore, for each $s \in \mathbb{R}$, let $H(s) = \displaystyle\int_0^s h(t) {\rm d}t$, and suppose that
\begin{description}
	\item[($A_1$)]\label{A1} $H(s)\geq 0$ for all $s \geq 0$;
	\item[($A_2$)]\label{A2} ${0}< \displaystyle\limsup_{s \to +\infty} \frac{H(s)}{s^p} < +\infty.$
\end{description}

In the spirit of Cammaroto, Chinnì and Di Bella \cite{Cammaroto},  one can prove the existence of  infinitely many weak solutions of problem \eqref{eq:D}, as follows:

\begin{theorem}\label{Appl:1}
Let $(M, F)$ be a noncompact, complete $n$-dimensional reversible Finsler manifold with  $\mathsf{Ric}_n \geq 0$, $0 < \mathsf{AVR}_F \leq 1$, and $2 \leq n < p < \infty$. Let $\Omega \subset M$ be a bounded open set with $C^1$ boundary, and let $h: \mathbb{R} \to \mathbb{R}$ be a continuous function with $h(0) = 0$, such that $H$ verifies conditions \hyperref[A1]{$(A_1)$} -- \hyperref[A2]{$(A_2)$}. Furthermore, assume that there exist two sequences $\{a_k\}$ and $\{b_k\}$ in $(0, +\infty)$, such that $a_k < b_k$, $\lim_{k \to \infty} b_k = +\infty$, $\lim_{k \to \infty} \frac{b_k}{a_k} = +\infty$, and
$\max_{[a_k,b_k]} h \leq 0$, for all $k \in \mathbb{N}$.
Then, there exists $\lambda_0>0$ such that for every $\lambda>\lambda_0$,  problem \eqref{eq:D} admits an unbounded sequence of weak solutions in $W_{0,F}^{1,p}(\Omega)$.
\end{theorem}		

The proof is based on the following critical point result of  Ricceri \cite[Theorem 2.5]{Ricceri00}:

\begin{theorem} \label{Th:CriticalPoint}
Let $(X, \|\cdot\|)$ be a reflexive real Banach space, and let $\Phi, \Psi: X \to \mathbb{R}$ be two sequentially weakly lower semicontinuous and Gâteaux differentiable functionals, such that
$\Psi$ is $($strongly$)$ continuous and satisfies $\lim_{\|x\|\to +\infty}\Psi(x) = +\infty$. For every $r > \inf_X \Psi$, put
\begin{equation}\label{varphi-definicio}
	 \varphi(r) = \inf_{u \in \Psi^{-1}(-\infty, r)} \frac{\Phi(u) - \inf_{v \in \overline{(\Psi^{-1}(-\infty, r))_w}} \Phi(v)}{r - \Psi(u)},
\end{equation}
where $\overline{(\Psi^{-1}(-\infty, r))_w}$ is the closure of $\Psi^{-1}(-\infty, r)$ in the weak topology. Let $\lambda \in \mathbb{R}$ be fixed. If $\{r_k\}$ is a real sequence such that $\lim_{k \to \infty} r_k = +\infty$ and $\varphi(r_k) < \lambda$ for all $k \in \mathbb{N}$, then either $\Phi + \lambda \Psi$ has a global minimum, or there exists a sequence $\{u_k\}$ of critical points of $\Phi + \lambda \Psi$ such that $\lim_{k \to \infty} \Psi(u_k) = +\infty$.
\end{theorem}

\textit{Proof of Theorem \ref{Appl:1}.}
We shall apply Theorem \ref{Th:CriticalPoint} by choosing $X = W_{0,F}^{1,p}(\Omega)$ endowed with the norm 
$$\|u\| = \left(\int_\Omega F^*(x, Du(x))^p \mathrm{d}v_F \right)^{1/p}.$$
As $\Omega$ is bounded, by Theorem \ref{theorem-1-intro} it follows that there exists a constant $c>0$ such that for any  $u \in  W_{0,F}^{1,p}(\Omega)$, 
\begin{equation}\label{pointwise}
	\sup_{x \in \Omega} |u(x)| \leq c \|u\|. 
\end{equation} Moreover, the embedding 
\begin{equation} \label{eq:embedding}
W_{0,F}^{1,p}(\Omega) \hookrightarrow L^\infty(\Omega)
\end{equation}
is compact, which follows from  Hebey \cite{Hebey} and the equivalence of the Finsler metric $F$ to any complete Riemannian metric on the bounded $\Omega$.  

Let $\Phi,\Psi:W_{0,F}^{1,p}(\Omega)\to \mathbb R$ be defined by 
$$\Phi(u) = -\int_\Omega H(u(x)) \mathrm{d}v_F 
\quad  \text{and} \quad
\Psi(u) = \int_\Omega F^*(x, Du(x))^p \mathrm{d}v_F, $$
and for $\lambda>0 $ we consider the energy functional associated with problem \eqref{eq:D} as 
$$ \mathcal{E}_\lambda: W_{0,F}^{1,p}(\Omega) \to \mathbb R, \quad \mathcal{E}_\lambda(u) = \lambda \Phi(u) + \frac{1}{p} \Psi(u). $$
Then, the critical points of $\mathcal E_\lambda$ are precisely
the weak solutions of problem \eqref{eq:D}. Standard arguments based on the compact
embedding \eqref{eq:embedding} imply that the functionals $\Phi$ and $\Psi$ are sequentially weakly lower semicontinuous and Gâteaux differentiable. Also, as a norm-type function, $\Psi$  is (strongly) continuous and coercive.
Furthermore, for each $r > 0$, the function $\varphi$ in \eqref{varphi-definicio} takes the form
$$ \varphi(r) = \inf_{\|u\|^p < r} \frac{\displaystyle\sup_{\|v\|^p \leq r} \displaystyle\int_\Omega H(v(x)) \mathrm{d}v_F - \displaystyle\int_\Omega H(u(x)) \mathrm{d}v_F }{r - \|u\|^p}.$$
Following the arguments of Cammaroto, Chinnì and Di Bella \cite[Theorem 1.1]{Cammaroto}, by taking $r_k = \left(\frac{b_k}{c}\right)^p$ (where $c>0$ comes from \eqref{pointwise}), it can be shown that $\varphi(r_k) < \frac{1}{p}$ for every  $k \in \mathbb{N}$, where the crucial step is the point-wise estimate  \eqref{pointwise}. 
Finally, one can prove that there exists $\lambda_0>0$ such that for every $\lambda>\lambda_0$, the functional $\mathcal{E}_\lambda$ is not bounded from below in $W_{0,F}^{1,p}(\Omega)$. Therefore,  Theorem \ref{Th:CriticalPoint} yields the existence of a sequence $\{u_k\} \subset W_{0,F}^{1,p}(\Omega)$ of critical points of $\mathcal{E}_\lambda$ such that $\lim_{k \to \infty} \|u_k\| = +\infty$.  \hfill $\square$

\subsection{Existence of a nonzero solution for a Dirichlet problem involving a singular term.}

As an application of Theorem \ref{Th:BrezisPoincareVazquez}, we consider on $(M, F)$ the following
semilinear Dirichlet problem 
\begin{align}\label{eq:P}
\begin{cases}
-\Delta_F u(x) - \mu \displaystyle \frac{u(x)}{d_F(x_0, x)^2} + \lambda u(x) = \displaystyle|u(x)|^{p-2} u(x) , \quad x \in \Omega \\
u \geq  0, ~ u \in W_{0,F}^{1,2}(\Omega) ,
\end{cases}
\tag{$\mathcal{P}_{\mu, \lambda}$}
\end{align}
where $\Delta_F$ is the $2$-Finsler-Laplace operator on $(M, F)$, $\Omega \subset M$ is a bounded, open set with $C^1$ boundary, and $x_0 \in \Omega$ is arbitrarily fixed. In addition, suppose that $p \in (2,2^*)$, $2^*$ being the critical
Sobolev exponent, i.e., $2^* =  2n/(n-2)$ if $n \geq 3$ and $2^* =  +\infty$ if $n = 2$.

If $\mu$ and $\lambda \in \mathbb{R}$ belong to a suitable range of parameters, one can show the existence of a nonzero
solution of problem \eqref{eq:P}, namely:  

\begin{theorem} \label{Application2}
	Let $(M, F)$ be a noncompact, complete $n$-dimensional reversible Finsler manifold with  $\mathsf{Ric}_n \geq 0$, $0 < \mathsf{AVR}_F \leq 1$, and $n \geq 2$. Let $\Omega\subset M$ be a bounded open set with $C^1$ boundary, $x_0 \in \Omega$ and $p \in (2,2^*)$. If either $\mu = 0$ when $n=2$, or $\mu \in \left[0, \frac{(n-2)^2}{4}\mathsf{AVR}_F^\frac{2}{n}\right)$ when $n \geq 3$, and $\lambda > - S_{\mu,F}(\Omega)$,  $S_{\mu,F}(\Omega)$ being the constant given by \eqref{S_mu_F}, then problem \eqref{eq:P} has a nontrivial and  nonnegative weak solution. 
	
	\begin{proof}
	Suppose that $\mu = 0$ when $n=2$, or $\mu \in \left[0, \frac{(n-2)^2}{4}\mathsf{AVR}_F ^\frac{2}{n}\right)$ when $n \geq 3$. Let $\lambda > - S_{\mu,F}(\Omega)$, and let us define the number $c_{\mu, \lambda} \in (0,1]$ by
	\begin{equation*}
	c_{\mu, \lambda} \coloneqq 
	\begin{cases} 
	\min\left(1, 1+\displaystyle\frac{\lambda}{S_{\mu,F}(\Omega)}\right), & \text{if } n = 2 \\
	\displaystyle\frac{4}{(n-2)^2} \overline \mu^2 \min\left(1, 1+\displaystyle\frac{\lambda}{S_{\mu,F}(\Omega)}\right), & \text{if } n \geq 3  
	\end{cases},
	\end{equation*}
	where $S_{\mu,F}(\Omega)$ and $\overline \mu$ are given by \eqref{S_mu_F}.	
	
	Then, using the \textit{Hardy} inequality \eqref{Hardy-intro} and the \textit{Brezis-Poincar\'e-V\'azquez} inequality \eqref{BVP-1}, it turns out that for every $ u \in W_{0,F}^{1,2}(\Omega)$, we have  
	\begin{equation*}
	\mathcal{K}^2_{\mu, \lambda}(u) \coloneqq \int_{\Omega} \Big\{ F^*(x, D u(x))^2  - \mu  \frac{u(x)^2}{d_F(x_0, x)^2} + \lambda u(x)^2 \Big\} \mathrm{d}v_F  
	\geq c_{\mu, \lambda} \int_{\Omega} F^*(x, D u(x))^2 \mathrm{d}v_F.
	\end{equation*}
		Therefore,  the functional $u \mapsto \mathcal{K}_{\mu,\lambda}(u)$ defines a norm on $W_{0,F}^{1,2}(\Omega)$, which is equivalent to the usual Dirichlet-norm $\|\cdot\|_{D_0^1}$.
	
	We associate with problem \eqref{eq:P} its energy functional $\mathcal{E}_{\mu,\lambda}: W_{0,F}^{1,2}(\Omega) \to \mathbb{R}$ defined by 
	\begin{equation*}
	\mathcal{E}_{\mu,\lambda}(u) \coloneqq \frac{1}{2} \mathcal{K}^2_{\mu,\lambda}(u) - \int_\Omega G(u(x))  \mathrm{d}v_F,
	\end{equation*}
	where $G: \mathbb{R} \to [0,\infty)$, $G(s) = \frac{s_+^p}{p}$ and $s_+ = \max(0,s)$.	 
	In a standard manner one can prove that $\mathcal{E}_{\mu,\lambda} \in C^1(W_{0,F}^{1,2}(\Omega);\mathbb{R})$. Moreover, $\mathcal{E}_{\mu,\lambda}$ verifies the conditions of the mountain pass theorem. Indeed, for any $p \in(2,2^*)$, by the Sobolev embedding theorem the continuous embedding $W_{0,F}^{1,2}(\Omega) \subset L^p(\Omega)$ holds, thus there exists a constant $c_1 >0$ such that $\|u\|_{L^p} \leq c_1 \|u\|_{W^{1,2}_0}$ for every $u\in W_{0,F}^{1,2}(\Omega)$. Furthermore, as $\Omega$ is a bounded domain, the norms $\|\cdot\|_{W^{1,2}_0}$ and $\|\cdot\|_{D^1_0}$ are equivalent, therefore we obtain that there exists a constant $c_2 = c_2(c_1,p) > 0$ such that	
	$$\mathcal{E}_{\mu,\lambda}(u) \geq \frac{1}{2} \mathcal{K}^2_{\mu,\lambda}(u)  -  \frac{c_1^p}{p} \|u\|^p_{W^{1,2}_0} \geq \mathcal{K}^2_{\mu,\lambda}(u) \left(\frac{1}{2} - c_2 \cdot \mathcal{K}^{p-2}_{\mu,\lambda}(u)\right), $$
	thus there exists a sufficiently small 
	$\rho >0$ such that 
	$$\inf_{\mathcal{K}_{\mu,\lambda}(u) = \rho} \mathcal{E}_{\mu,\lambda}(u) > 0 = \mathcal{E}_{\mu,\lambda}(0) .$$
	
	Furthermore, for any $t > 0$ and $u \in W_{0,F}^{1,2}(\Omega)$ with $u \geq 0$, we have $$\mathcal{E}_{\mu,\lambda}(t u) = \frac{t^2}{2} \mathcal{K}^2_{\mu,\lambda}(u) - t^p \int_\Omega G(u(x))  \mathrm{d}v_F \leq \frac{t^2}{2} \int_{\Omega} F^*(x, D u(x) )^2 \mathrm{d}v_F  - \frac{t^p}{p} \int_\Omega u(x)^p \mathrm{d}v_F ,$$
	thus there exists a sufficiently large $t>0$ and $\overline{u} \in W_{0,F}^{1,2}(\Omega)\setminus \{0\}$, $\overline{u} \geq 0$ such that $\mathcal{K}_{\mu,\lambda}(t\overline{u}) > \rho$ and  $\mathcal{E}_{\mu,\lambda}(t \overline{u}) \leq 0 $.
	
	Since $W_{0,F}^{1,2}(\Omega)$ is compactly embedded into $L^p(\Omega)$ for any $p \in [2,2^*)$, it can be proven that $\mathcal{E}_{\mu,\lambda}$ satisfies the Palais-Smale condition at each level. Therefore, by the Ambrosetti-Rabinowitz theorem (see Willem \cite[Lemma 1.20]{Willem}), it follows that $\mathcal{E}_{\mu,\lambda}$ has a positive critical value corresponding to a nontrivial weak solution $u \in W_{0,F}^{1,2}(\Omega)$ of the problem 
$$
	\begin{cases}
	-\Delta_F u(x) - \mu \displaystyle\frac{u(x)}{d_F(x_0, x)^2} + \lambda u(x) = u(x)_+^{p-1} , \quad x \in \Omega \\
	u \in W_{0,F}^{1,2}(\Omega) .
	\end{cases}
$$
	Multiplying the first equation by $u_-(x) = \min(0,u(x))\in  W_{0,F}^{1,2}(\Omega)$ and integrating over $\Omega$, we obtain that $\mathcal{K}^2_{\mu,\lambda}(u_-) = 0$, which in turn yields that $u_- = 0$. Thus $u \geq 0$ is a nontrivial solution to problem \eqref{eq:P}. 
	\end{proof}
\end{theorem}


\vspace{0.5cm}

\begin{thebibliography}{99}
		
	\bibitem{AIHP-Lions} A. Alvino, V. Ferone, G. Trombetti, P.-L. Lions, Convex
	symmetrization and applications. \textit{Ann. Inst. H. Poincar\'{e} Anal. Non Lin\'{e}aire} 14 (1997), no. 2, 275--293.

	\bibitem{Aubin} 
	T. Aubin, Probl\`emes isop\'erim\'etriques et espaces de Sobolev. \textit{J. Differential Geometry} 11 (1976), no. 4, 573--598.

	\bibitem{BalinskyEvansLewis} A.~A. Balinsky, W.~D. Evans,  R.~T. Lewis, \textit{The Analysis and Geometry of Hardy's Inequality}. Universitext, Springer-Verlag, 2015.
	
	\bibitem{BaloghKristaly} Z. M. Balogh, A. Krist\'aly, Sharp isoperimetric and Sobolev inequalities in spaces with nonnegative Ricci curvature. \textit{Math. Ann.}, in press.  DOI: https://doi.org/10.1007/s00208-022-02380-1.
	
	\bibitem{BCS}
	D.~Bao, S.-S.~Chern,  Z.~Shen, 
	\textit{An Introduction to Riemann-Finsler Geometry}. 
	Graduate Texts in Mathematics Vol.~200, Springer-Verlag, 2000.
	
	
	\bibitem{Berchio-etal} E. Berchio, D. Ganguly, G. Grillo, Y. Pinchover,
	An optimal improvement for the Hardy inequality on the hyperbolic space and related manifolds. 
	\textit{Proc. Roy. Soc. Edinburgh Sect. A} 150 (2020), no. 4, 1699--1736.
	
	

	\bibitem{Brendle} S. Brendle, Sobolev inequalities in manifolds with nonnegative curvature. \textit{Comm. Pure. Appl. Math.},  to appear. 
	
\bibitem{BZ}	J. E. Brothers, W. P. Ziemer, Minimal rearrangements of Sobolev functions. \textit{J. Reine Angew. Math.} 384
	(1988), 153--179.

	\bibitem{CabreRosOton} 
	X. Cabr\'e, X. Ros-Oton, J. Serra,
	Sharp isoperimetric inequalities via the ABP method.
	\textit{J. Eur. Math. Soc.} 18 (2016), 2971--2998.

	\bibitem{Cammaroto}
	F. Cammaroto, A. Chinnì, B. Di Bella, Infinitely many solutions for the Dirichlet problem involving the $p$-Laplacian. \textit{Nonlinear Anal.} 61 (2005), 41--49.
	
	\bibitem{olaszok} L. D'Ambrosio, S. Dipierro, Hardy inequalities on Riemannian manifolds and applications. \textit{Ann. Inst. H. Poincar\'e Anal. Non Lin\'eaire} 31 (2014), no. 3, 449--475.
	
	
%	\bibitem{Carron} G. Carron, Euclidean volume growth for complete Riemannian manifolds.  \textit{Milan J. Math.} 88 (2020), 455--478.
%	
%	\bibitem{Carron2} G. Carron, Inegalit\'es isop\'erim\'etriques de Faber-Krahn et consequences. \textit{Publications de
%		l' Institut Fourier}, 220, 1992.
%	
%	\bibitem{Cavalletti-Mondino} F. Cavalletti, A. Mondino,
%	Sharp and rigid isoperimetric inequalities in metric-measure spaces with lower Ricci curvature bounds. 
%	\textit{Invent. Math.} 208 (2017), no. 3, 803--849. 
%	
%	
%	\bibitem{Chavel} I. Chavel,  Riemannian Geometry. A Modern Introduction.
%	Second Edition. 2006.
%	
%	
%	\bibitem{CC} J. Cheeger, T. H.  Colding, 
%	Lower bounds on Ricci curvature and the almost rigidity of warped products.
%	\textit{Ann. of Math.} (2) 144 (1996), no. 1, 189--237. 
%	
%
%	
	%\bibitem{Croke} C. Croke, A sharp four-dimensional isoperimetric inequality.
	%\textit{Comment. Math. Helv.} 59 (1984), no. 2, 187--192.
	
%	\bibitem{Colding} T. H. Colding, 
%	Ricci curvature and volume convergence.
%	\textit{Ann. of Math.} (2) 145 (1997), no. 3, 477--501. 
%	
%	\bibitem{CEMS} D. Cordero-Erausquin, R. J. McCann, M. Schmuckenschl\"{a}ger, 
%	A Riemannian interpolation inequality \`{a} la Borell, Brascamp and Lieb, \textit{Invent. Math.} 146, (2001), 219--257.
%	
%	\bibitem{CE-N-Villani}  D. Cordero-Erausquin, B. Nazaret, C. Villani, A mass-transportation approach to sharp Sobolev and Gagliardo-Nirenberg inequalities. \textit{Adv. Math.} 182 (2004), no. 2, 307--332.
%	
%	
%	\bibitem{DelPino-Dolb} M. Del Pino, J. Dolbeault, The optimal Euclidean $L^p$-Sobolev logarithmic inequality. \textit{J. Funct. Anal.} 197 (2003)
%	151--161.
%	
%	%\bibitem{DH} O. Druet, E. Hebey, The AB program in geometric
%	%analysis: sharp Sobolev inequalities and related problems. Mem.
%	%Amer. Math. Soc. 160 (2002), no. 761.
	
	\bibitem{Druetetal} O. Druet, E. Hebey, M. Vaugon, Optimal Nash's inequalities on
	Riemannian manifolds: the influence of geometry.  \textit{Int. Math. Res. Not.} 14 (1999), 735--779.
	
	\bibitem{doCarmo-Xia} M. P. do Carmo, C. Xia,
	Complete manifolds with nonnegative Ricci curvature and the
	Caffarelli-Kohn-Nirenberg inequalities. \textit{Compos. Math.} {140} (2004),
	818--826.

%		\bibitem{FarkasKristalyVarga_Calc}
%		C.~Farkas, A.~Krist\'{a}ly, C.~Varga,
%		\newblock Singular {P}oisson equations on {F}insler-{H}adamard manifolds.
%		\newblock \textit{Calc. Var. Partial Differential Equations} 54(2):1219--1241, 2015.
		
		
	\bibitem{GS} M. Ghomi, J. Spruck,  Total curvature and the isoperimetric inequality in Cartan-Hadamard manifolds. \textit{J. Geom. Anal.} 32 (2022), no. 2, Paper No. 50, 54 pp.
	
	\bibitem{Hebey} E. Hebey, \textit{Nonlinear analysis on manifolds: Sobolev spaces and inequalities}. Courant Lecture Notes in Mathematics, 5.
	New York University, Courant Institute of Mathematical Sciences, New
	York; American Mathematical Society, Providence, RI, 1999.	

\bibitem{HKZ} L. Huang, A. Krist\'aly, W. Zhao,    Sharp uncertainty principles on general Finsler manifolds. \textit{Trans. Amer. Math. Soc.} 373 (2020), no. 11, 8127-8161.


\bibitem{KK} B.R. Kloeckner, G.   Kuperberg, 
The Cartan-Hadamard conjecture and the Little Prince. 
\textit{Rev. Mat. Iberoam.} 35 (2019), no. 4, 1195--1258.

%	
%	\bibitem{Kristaly-Calculus} A. Krist\'aly,  Metric measure spaces supporting Gagliardo-Nirenberg inequalities: volume non-collapsing and rigidities. \textit{Calc. Var. Partial Differential Equations} 55 (2016), no. 5, Art. 112, 27 pp.	
	\bibitem{Kristaly-Potential} A. Krist\'aly,   Sharp Morrey-Sobolev inequalities on complete Riemannian manifolds. \textit{Potential Anal.} 42 (2015), no. 1, 141--154.
%	

	
	\bibitem{Kristaly-JMPA} A. Krist\'aly, Sharp uncertainty principles on Riemannian manifolds: the influence of curvature,  \textit{J. Math. Pures. Appl. (Liouville Journal)}, 119 (2018), 326--346. 
	
	\bibitem{Kristaly-Calculus} A. Krist\'aly,  Metric measure spaces supporting Gagliardo-Nirenberg inequalities: volume non-collapsing and rigidities. \textit{Calc. Var. Partial Differential Equations} 55 (2016), no. 5, Art. 112, 27 pp.
	
		\bibitem{Kri-Ohta} A. Krist\'aly, S. Ohta, Caffarelli-Kohn-Nirenberg inequality on metric measure spaces with applications. \textit{Math. Ann.} 357 (2013), no. 2, 711--726.

	\bibitem{Kristaly-JDE} A. Krist\'aly, A. Szak\'al, Interpolation between Brezis-V\'azquez and Poincar\'e inequalities on nonnegatively curved spaces: sharpness and rigidities. \textit{J. Differential Equations} 266 (2019), no. 10, 6621--6646.
	
	
	
	
	%\bibitem{Ledoux} M. Ledoux, Isoperimetry and Gaussian analysis, Lectures on Probability Theory and Statistics
	%(Saint-Flour, 1994), Lecture Notes in Mathematics, Vol. 1648,
	%Springer, Berlin, 1996, pp. 165--294.
	
	\bibitem{Ledoux-CAG} M. Ledoux, On manifolds with nonnegative Ricci curvature and Sobolev inequalities.
	\textit{Comm. Anal. Geom.} 7 (1999), no. 2, 347--353.
	
	\bibitem{LL} E. H. Lieb, M. Loss, \textit{Analysis. Second edition.} Graduate Studies in Mathematics, 14. American Mathematical Society, Providence, RI, 2001.
	

\bibitem{MS} M. Muratori, N. Soave, Some rigidity results for Sobolev inequalities and related PDEs on Cartan-Hadamard manifolds. \textit{Annali della Scuola Normale Superiore di Pisa, Classe di Scienze}, to appear. 
	
	%\bibitem{Moser} J. Moser, A sharp form of an inequality by N. Trudinger. \textit{Indiana Univ. Math.} J. 20 (1971) 1077--1091.
	

	\bibitem{Ohta}  S. Ohta, Finsler interpolation inequalities. \textit{Calc. Var. Partial Differential Equations} 36 (2009),
	211--249.


	\bibitem{Ohta-Sturm} S.~Ohta, K.-T.~Sturm, 
	Heat flow on Finsler manifolds. \textit{Comm. Pure Appl. Math.}
	62, no. 10, (2009), 1386--1433.
	
	\bibitem{Petersen} P. Petersen, \textit{Riemannian geometry}, Third edition, Graduate Texts in Mathematics, Springer, 2016.

	\bibitem{Ricceri00} B. Ricceri, A general variational principle and some of its applications. \textit{J. Comput. Appl. Math.} 113 (2000), 401--410.

	\bibitem{Shen01} Z. Shen, \textit{Lectures on Finsler geometry}. World Scientific Publishing Co., Singapore, 2001.
	
	
	\bibitem{Shen-volume} Z. Shen, Volume comparison and its applications in Riemann-Finsler geometry. \textit{Adv. Math.} 128 (1997), 306--328.

	
%	\bibitem{Sturm-1} K.-T. Sturm, On the geometry of metric measure spaces. I, \textit{Acta Math.} 196 (1) (2006), 65--131.
%	
%	\bibitem{Sturm-2} K.-T. Sturm, On the geometry of metric measure spaces. II, \textit{Acta Math.} 196 (1) (2006), 133--177.
		
	\bibitem{Talenti} G. Talenti, Inequalities in rearrangement invariant function spaces. Nonlinear Analysis, Function Spaces and Applications, \textit{Proceedings of the Spring School held in Prague, May 23-28, 1994.} {V}ol. 5. (1994), 177--230.

	\bibitem{Willem} M. Willem, 
	\textit{Minimax theorems}.
	Progress in Nonlinear Differential Equations and their Applications, 24. Birkh\"auser
	Boston, Inc., Boston, MA, 1996.
	
%	\bibitem{Wu-Xin} B.Y. Wu, Y.L. Xin, 
%	Comparison theorems in Finsler geometry and their applications. 
%	\textit{Math. Ann.} 337 (2007), no. 1, 177--196.
	
	%\bibitem{Weil} A. Weil, Sur les surfaces a courbure n\'egative. C. R. Acad. Sci. Paris S\'er I Math. 182
	%(1926), 1069--1071.
	
	%\bibitem{Weissler} F. B. Weissler, Logarithmic Sobolev inequalities for the heat-diffusion semigroup. Trans. Amer. Math. Soc. 237
	%(1978), 255--269.
	
	
	%\bibitem{Zhao-Shen} W. Zhao, Y. Shen, A universal volume comparison theorem
	%for Finsler manifolds and related results. Canad. J. Math. 65
	%(2013), no. 6, 1401--1435.
	
	 \bibitem{Zhao} W. Zhao,  Hardy inequalities with best constants on Finsler metric measure manifolds. \textit{J. Geom. Anal.} 31 (2021), no. 2, 1992--2032.
	
	
\end{thebibliography}
\end{document}